\documentclass[12pt,twoside]{article}

\usepackage[T1]{fontenc}
\usepackage[latin1]{inputenc}
\usepackage[english]{babel}
\usepackage[babel]{csquotes}

\usepackage{cite}

\usepackage{amssymb}
\usepackage{amsmath}
\usepackage{amsthm}
\usepackage{latexsym}
\usepackage{graphicx}
\usepackage{mathrsfs}
\usepackage{mathrsfs}

\usepackage[usenames,dvipsnames]{color}

\definecolor{viola}{rgb}{0.3,0,0.7}
\definecolor{ciclamino}{rgb}{0.5,0,0.5}

\def\pier #1{{\color{red}#1}}
\def\pier #1{#1}

\def\luca #1{{\color{red}#1}}
\def\luca #1{#1}

\theoremstyle{plain}
\newtheorem{thm}{Theorem}[section]
\newtheorem{lem}[thm]{Lemma}

\theoremstyle{definition}
\newtheorem{rmk}[thm]{Remark}

\def\eps{\varepsilon}
\def\weakstar{\stackrel{*}{\rightharpoonup}}
\def\En{\mathbb{N}}

\def\Ar{\mathbb{R}}

\def\Pi{\mathbb{P}}

\def\beq{\begin{equation}}
\def\eeq{\end{equation}}

\def\rarr{\rightarrow}

\def\l|{\left\|}
\def\r|{\right\|}

\def\L2{L^2(\Omega)}
\def\H1{H^1(\Omega)}
\def\gH1{H^1_{0,\Gamma_b}(\Omega)}
\def\nL2#1{\l|#1\r|_{\L2}}
\def\nH1#1{\l|#1\r|_{\H1}}
\def\rarrw{\rightharpoonup}
\def\Div{\mathop{\rm div}}

\def\gL2{L^2(\Gamma_1)}
\def\ngL2#1{\l|#1\r|_{\gL2}}

\begin{document}
\begin{center}
\pier{%
{\huge\rm Existence of solutions for a model\\[0.2cm] of microwave heating\/\footnote{{\bf 
Acknowledgments.}\quad\rm The first author gratefully acknowledges some 
financial support from the MIUR-PRIN Grant 2010A2TFX2 ``Calculus of Variations'', the 
GNAMPA (Gruppo Nazionale per l'Analisi Matematica, la Probabilit\`a e le loro 
Applicazioni) of INdAM (Istituto Nazionale di Alta Matematica) and the IMATI -- C.N.R. 
Pavia.}}
\\[0.5cm]
{\large\sc Pierluigi Colli$^{(1)}$}\\
{\normalsize e-mail: {\tt pierluigi.colli@unipv.it}}\\[.25cm]
{\large\sc Luca Scarpa$^{(1),\, (2)}$}\\
{\normalsize e-mail: {\tt luca.scarpa01@ateneopv.it}}\\[.25cm]
$^{(1)}$
{\small Dipartimento di Matematica ``F. Casorati'', Universit\`a di Pavia}\\
{\small via Ferrata 1, 27100 Pavia, Italy}\\[.2cm]
$^{(2)}$
{\small Department of Mathematics, \luca{University College London}}\\
{\small \luca{Gower Street, London WC1E 6BT,} United Kingdom}}\\[1cm]
\end{center}       

\begin{abstract}
This paper is concerned with a system of differential equations related to a circuit model for 
microwave heating, complemented by suitable initial and boundary conditions. A RCL circuit with a 
thermistor is representing the microwave heating process with temperature-induced modulations 
on the electric field. The unknowns of the PDE system are the absolute temperature in the body, 
the voltage across the capacitor and the electrostatic potential. Using techniques based on 
monotonicity arguments and sharp estimates, we can prove the existence of a weak solution 
to the initial-boundary value problem.\\[.5cm]
\end{abstract}

\pagestyle{myheadings}
\newcommand\testopari{\sc \pier{Pierluigi Colli and Luca Scarpa}}
\newcommand\testodispari{\sc \pier{Existence of solutions for a model of microwave heating}}
\markboth{\testodispari}{\testopari}


\thispagestyle{empty}

\section{Introduction}
\setcounter{equation}{0}
\label{intro}

In this work, we deal with a problem which arises from a circuit model for microwave heating: in particular, we aim at proving the existence of a solution
to a \pier{coupled} system of three differential equations (an ODE, an elliptic equation 
and a nonlinear parabolic PDE) and appropriate initial and boundary conditions.
More specifically, we consider a traditional RLC circuit in which a thermistor has been inserted: this one has a cylindrical shape and takes into account the
temperature's effect (for further details see  \cite{reim-jor-min-var}). The system of equations we focus on is obtained in \cite{reim-jor-min-var}: it involves the \pier{absolute} temperature 
$\vartheta$, the voltage $V$ across the capacitor and the potential $\phi$. 

Firstofall, let us introduce the notation that we will use in the following:
\pier{%
\begin{equation}
  \label{omega}
  B\subseteq\Ar^2\text{ smooth bounded domain}\,, \quad \Omega=B\times(0,\pier{\ell})\,,
\end{equation} 
with $\ell>0$ being the height of the cylinder;  
\begin{equation}
  \label{gamma_dom}
  \Gamma=\partial\Omega= \Gamma_{\pier{\ell}}\cup \Gamma_b \,, \quad \Gamma_{\pier{\ell}}=\partial B\times(0,\pier{\ell})\,,
  \quad \Gamma_b=B\times\{0,l\}\,,\\
\end{equation}
 so that $\Gamma_{\pier{\ell}}$ denotes the lateral boundary and $\Gamma_b$ collects the 
 basis and top boundaries of the cylinder; 
  \begin{equation}
  \label{Q}
  \quad Q=\Omega\times(0,T)\,,  \quad \Sigma_{\pier{\ell}}=\Gamma_{\pier{\ell}}\times(0,T)\,, \quad \Sigma_b=\Gamma_b\times(0,T)\,,
\end{equation}
where  $T>0$ denotes the final time. Hence,} the thermistor is represented by the cylinder $\Omega$ in $\Ar^3$, while $Q$ represents the \pier{spatiotemporal} domain.

The system obtained in \cite{reim-jor-min-var}, \pier{complemented} with initial and boundary conditions, is the following:
\begin{gather}
  \label{system1}
  CV''(t)+\frac{1}{R}V'(t)+\frac{1}{L}V(t)=-\frac{d}{dt}\int_B{\sigma(\vartheta(t))\frac{\partial\phi}{\partial z}(t)\,dx}+f(t)  \quad \pier{\text{for }  t\in[0,T]}\,,\\
  \label{system2}
 \pier{ \Div\bigl(\sigma(\vartheta(t))\nabla\phi(t) \bigr)=0 \quad\text{in } \Omega, \pier{\text{ for }  t\in[0,T]}}\,,\\
  \label{system3}
  c_0\vartheta_t - \Div\bigl(k(\vartheta)\nabla\vartheta\bigr)= \sigma(\vartheta)\left|\nabla\phi\right|^2 \quad\text{in } Q\,,\\
  \label{init1}
  V(0)=V_0\pier{\ \hbox{ and }\ }  V'(0)=V'_0\,, \quad \vartheta(0)=\vartheta_0\  \, \pier{\text{in } \Omega} \,,\\
  \label{bound1}
  \pier{ \phi (t) =0 \, \text{ \pier{on} } B\times\{0\}\ \hbox{ and }\  \phi(t) =V (t) \, \text{ \pier{on} } B\times\{\ell\}\,,  \  \pier{\text{for }  t\in[0,T]}\,,}
  \\
  \label{pier1}
  \quad \pier{\sigma(\vartheta(t) )\nabla\phi (t) \cdot{\bf n} }=0 \, 
  \text{ \pier{on} } 
  \Gamma_{\pier{\ell}}, \pier{\text{ for }  t\in[0,T]}\,,\\
  \label{bound2}
 \pier{{}-{}} k(\vartheta)\nabla\vartheta\cdot{\bf n}=0 \, \text{ \pier{on} } \Sigma_b\,, 
  \quad \pier{{}-{}} k(\vartheta)\nabla\vartheta\cdot{\bf n}=h(\vartheta)-h(\vartheta_{\Gamma})\,
    \text{ \pier{on} } 
  \Sigma_{\pier{\ell}}\,,
\end{gather}
\pier{where $C,R,L>0$ are the capacitance, resistance and inductance coefficients, respectively, 
$f: (0,T) \to \Ar $ is a prescribed current source, $\sigma$ and $k$ represent the thermistor 
and heat conductivities and may vary with the temperature, $\vartheta_\Gamma$ denotes the 
known environment's temperature, and $h$ is a given increasing and continuous function with 
$h(0)=0$ .  
The notation ${\bf n} $ stands for the outward normal unit vector so that, in particular,
${\bf n} = (0,0,1) $ on $B\times\{\ell\}$ and  ${\bf n} = (0,0,-1) $ on $B\times\{0 \}$.}

\pier{Let us point out that throughout the paper we will use the notation $x$ for the variable in the two-dimensional domain $B$,
while $z$ will denote the third variable ranging in $(0,\ell)$.}
It is also important to \pier{emphasize} that the term
  \beq
    \label{IR}
    I_R(z,t):=\int_B{\sigma(\vartheta(t))\frac{\partial\phi}{\partial z}(t)\,dx}\,,
  \eeq
  which appears in equation \eqref{system1}, \pier{in principle could depend on both $z$ and $t$}: actually, we will \pier{check} that $I_R$ only depends on $t$, i.e.,
  \beq
    \label{z_der}
    \frac{\partial I_R}{\partial z}=0\,,
  \eeq
  so that $I_R=I_R(t)$ and \pier{the ordinary differential} equation \eqref{system1} makes sense.

The system \eqref{system1}--\eqref{bound2} \pier{turns out to be} interesting from a physical point of view. In fact, $V$ represents the voltage across the capacitor
and equation \eqref{system1} describes how this one is linked with  \pier{$\int_\Omega{\sigma(\vartheta)\frac{\partial\phi}{\partial z}\,dx}$, which renders 
the current across the thermistor. Equation \eqref{system2} is supplied with the mixed boundary conditions of Dirichlet type at $\Gamma_b$ (indeed, \eqref{bound1} involves the value of $V$ as well on the upper face) and the no-flux condition \eqref{pier1} across the boundary  $\Gamma_{\pier{\ell}}$, 
for $t$ varying in $(0,T)$. The third equation \eqref{system3} 
illustrates the evolution of the temperature with respect to the source of ohmic heating $\sigma(\vartheta)\left|\nabla\phi\right|^2$.
Moreover, the condition \eqref{bound2} is specifying the heat flux across the thermistor and, in particular, states the \luca{proportionality} with the difference of the values of the function $h$ on the inside and outside temperatures, at the boundary  $\Sigma_{\pier{\ell}}$. }

Beyond the relevance from the physical point of view, the problem \eqref{system1}--\eqref{bound2} is \pier{intriguing for} a mathematical approach:
as a matter of fact, it consists of an ordinary differential equation in $V$, an elliptic equation in $\phi$ and a nonlinear parabolic partial differential equation in $\vartheta$.
\pier{Nonlinear terms are present in all the three equations and also in the boundary condition \eqref{bound2}.} In this work, the terms we will accurately \pier{deal with} and that are more difficult to handle are the conductivity $\sigma(\vartheta)$ in the second equation \eqref{system2} and the right hand side $\sigma(\vartheta)\left|\nabla\phi\right|^2$ in \eqref{system3}: actually, \pier{the former makes a strong coupling for the potential $\phi$} and the latter would not be \pier{found} to have so much regularity.

We mention now about some related literature for microwave heating. In the 
contribution~\cite{Yin} the time-harmonic Maxwell equations with temperature dependent coefficients in a domain occupied by a conducting medium are coupled with the heat equation controlling the temperature distribution induced by the electric field: this is discussed as a model for microwave heating and the existence of a global solution of the coupled nonlinear system is shown. The paper~\cite{MYS} deals with a mathematical model coupling Maxwell's equations to the enthalpy formulation of the Stefan problem: under suitable conditions on the material properties, the model is proved to admit a global weak solution. An optimization problem for a microwave/induction heating process is studied in \cite{WYT}: the control variable is the applied electric field on the boundary and the cost function is shaped in order the temperature profile at the final time has a  relative uniform distribution in the field.  A reduced model for the microwave heating of a thin ceramic slab is investigated in \cite{AK} and steady-state solutions and their linear stability properties are discussed. 

In conclusion, we briefly outline the contents of our paper. 

\pier{In Section~\ref{result}} we arrive at the statement of the \pier{noteworthy} theorem, which ensures that the problem \eqref{system1}--\eqref{bound2} has a solution: in particular, we will write some general variational formulations for the equations and we will state the precise existence result.

In \pier{Section~\ref{approx}}, we will focus on the proof of the theorem: the idea is to use a delay argument in equation \eqref{system2} and consider a truncation \pier{of the term  $\sigma(\vartheta)\left|\nabla\phi\right|^2$} in \eqref{system3}. More precisely, given $\tau>0$, we will will prove the existence of a solution for the \pier{approximating} problem in each interval $[n\tau,(n+1)\tau]$ by using the result contained in \cite{scarpa} and a fixed point argument at each iteration. A solution for the \pier{approximating} problem will then be obtained
"pasting" together accurately all the solutions obtained in every single interval.

\pier{Section~\ref{estimat} contains the} uniform estimates on the solution of the \pier{approximating problem that are helpful for the limit procedure}. The idea here will be to test equation \eqref{system3} by suitable functions \pier{we will later introduce}.

Finally, \pier{Section~\ref{limit}} collects the arguments \pier{we use to pass to the 
limit in the \pier{approximating} problem in order to recover a solution
for the original one: the main tools here are} some compactness results.


\section{The main result}
\setcounter{equation}{0}
\label{result}

In this section, we present the main existence theorem that will be proved in the \pier{paper}. 
Firstly, we introduce a general reformulation of problem \eqref{system1}--\eqref{bound2}.

\pier{We multiply \eqref{system2} by a test function $w \in \gH1$, with 
\beq
  \label{test}
  \gH1:=\{w\in\H1: \  w=0 \, \text{ on } \Gamma_b\}\,;
\eeq
then, taking into account \eqref{pier1} and integrating by parts,  we easily obtain
\beq
  \label{system2_bis_phi}
  \int_{\Omega}{\sigma(\vartheta(t))\nabla\phi(t)\cdot\nabla w\,dxdz}=0 
\eeq
for $t\in [0,T]$. Now,  in order to eliminate the presence of $V(t)$ in the boundary condition \eqref{bound1}, let us change the variable $\phi$ by introducing
\beq
  \label{psi}
  \psi(x,z,t) :=\phi(x,z,t) - \frac{z}{\pier{\ell}}V(t)\, , \quad (x,z,t) \in \Omega \times [0,T]\, ; 
\eeq
in this way, an easy computation shows that conditions \eqref{bound1}--\eqref{pier1} become
\beq
  \label{bound2_bis}
 \psi (t) =0 \, \text{ \pier{on} } \Gamma_b , 
  \quad \pier{\sigma(\vartheta(t) )\nabla\psi(t) \cdot{\bf n} }=0 \, 
  \text{ \pier{on} } 
  \Gamma_{\pier{\ell}}, \pier{\ \text{ for a.e. }  t\in[0,T]}\,,
\eeq
since the third component of $\bf n$ is null on the lateral boundary $\Gamma_{\pier{\ell}}$. 
Hence, the equality \eqref{system2_bis_phi} can be rewritten as
\beq
  \label{system2_bis_psi}
  \begin{split}
  \int_{\Omega}{\sigma(\vartheta(t))\nabla\psi(t)\cdot\nabla w\,dxdz}+
  \int_{\Omega}{\sigma(\vartheta(t))\frac{V(t)}{\pier{\ell}}\frac{\partial w}{\partial z}\,dxdz}=0\\ 
\text{for all } w\in \gH1\,, \  \text{for a.e. } t\in[0,T] \, .
  \end{split}
\eeq
Note that, for a fixed $t\in [0,T]$, if $\sigma(\vartheta(t))$  lies in $L^\infty (\Omega ) $ and is bounded from below by a positive constant, then definition~\eqref{test} and the Poincar\'e inequality allow us to conclude that
$$
  w \mapsto \int_{\Omega}{\sigma(\vartheta(t))|\nabla w|^2\,dxdz}
$$
yields an equivalent norm in  $\gH1$. Within this framework, 
it is not difficult to check that the Lax-Milgram lemma 
implies the existence of a unique $\psi(t)$ solving the variational equality \eqref{system2_bis_psi}}.%

 \pier{Next, as we have anticipated, let us explain why equation \eqref{system1} makes sense by checking \eqref{z_der}.  Letting $\zeta\in {\cal D}(0,\pier{\ell})$ and 
choosing $w(x,z)=\zeta(z)$, $x\in B$ and $z\in(0,\pier{\ell})$, in  \eqref{system2_bis_phi},  an easy calculation shows that
  \[
    \int_0^l\int_B{\sigma(\vartheta(t))\frac{\partial\phi}{\partial z}(t)\frac{\partial\zeta}{\partial z}\,dxdz}=0\,.
  \]
Therefore,  using the Fubini-Tonelli theorems and integrating by parts lead to 
\beq
    -\int_0^l{\frac{d}{dz}\bigl(\int_B{\sigma(\vartheta(t))\frac{\partial\phi}{\partial z}(t)\,dx}\bigr)}\zeta\,dz=0 \quad \forall\,\zeta\in {\cal D}(0,\pier{\ell})\,,
     \label{no_zdep'} 
\eeq
from which condition \eqref{z_der} follows. Of course, in the above argument, $t\in [0,T]$ is fixed.}

\pier{Now, let us integrate equation \eqref{system1} with respect to time and renominate the constants; in view of \eqref{init1}, we obtain
\beq
\begin{split}
  \label{system1_bis}
  \lambda_1V'(t)&+\lambda_2V(t)+\lambda_3\int_0^t{V(r)\,dr}=
  \lambda_1V'_0+\lambda_2V_0\\
  &-\int_B{\sigma(\vartheta(t))\frac{\partial\phi}{\partial z}(t)\,dx}+
  \int_B{\sigma(\vartheta_0)\frac{\partial\phi_0}{\partial z}\,dx}+\int_0^t{f(r)\,dr} \quad \forall\, t\in[0,T]\,,
\end{split}
\eeq
where $\lambda_1,\, \lambda_2,\, \lambda_3>0$ substitute $C, \, 1/R, \, 1/L$, respectively. Please note that in \eqref{system1_bis} $\phi_0$ denotes the element 
$$
\phi_0(x,z) = \psi_0(x,z) + V_0 \frac z\ell , \quad  (x,z)\in \Omega\, ,
$$
where $ \psi_0 \in \gH1$ is the solution of \eqref{system2_bis_psi} corresponding to $t=0$, that is,
\beq
\label{pier2}
 \begin{split}
  \int_{\Omega}{\sigma(\vartheta_0) \nabla \psi_0 \cdot\nabla w\,dxdz}
  +
  \int_{\Omega}{\sigma(\vartheta_0)\frac{V_0}{\pier{\ell}}\frac{\partial w}{\partial z}\,dxdz}
  =0 \ 
\text{ for all } w\in \gH1\, ,
  \end{split}
\eeq
and $\phi_0, \,  \psi_0$ may be considered as auxiliary initial values.}

It is now time to set some assumptions. About the functions $\sigma$ and $k$, we
assume that
\beq
  \label{bound_sig_k}\sigma, \, k \in C^0 (\Ar) , \quad 
 0 < \sigma_*\leq\sigma(r) \leq\sigma^* \, \hbox{ and }\,  0< k_*\leq k(r) \leq k^*\quad \forall \, r\in \Ar\,,
\eeq
for some positive constants  $\sigma_*,\, \sigma^*,\, k_*,\, k^* $. Note that, although these \luca{functions} work in principle only on $(0,+\infty) $ due to the physical meaning of the variable $\vartheta$, in case they can be easily extended with reasonable values to the whole of $\Ar$.  About the initial data, we let 
\beq 
\label{pier3}
    V_0, V'_0\in\Ar\,, \quad \vartheta_0\in\L2
\eeq    
so that     
\beq 
\label{pier4}
    \psi_0 \in \gH1 \, \hbox{ and } \, \phi_0 \in \H1 \ \hbox{ are well defined}. 
\eeq

At this point, let us comment on equation \eqref{system3} and introduce
\beq
  \label{K}
  K:\Ar\rarr\Ar\,, \quad K(r)=\int_0^r{k(\rho)\,d\rho}, \ \, r\in \Ar \, ;
\eeq
due to \eqref{bound_sig_k},  $K\in C^1(\Ar) $ is a bi-Lipschitz continuous function. Then,
\beq
  \label{gamma}
  \gamma:=K^{-1}:\Ar\rarr\Ar\, \hbox{ is Lipschitz continuous along with its inverse function} 
\eeq
and  equation \eqref{system3} can be rewritten as
\beq
  \label{system3_bis}
  \vartheta_t-\Delta K(\vartheta)= \sigma(\vartheta) |\nabla\phi |^2  \quad\text{in } Q\,,
\eeq
where we have taken $c_0=1$ to simplify the notation. A variational formulation of 
\eqref{system3_bis} can be easily obtained by multiplying it by a test function $w$ and 
integrating by parts on account of the boundary condition \eqref{bound2}. In particular, if we introduce the auxiliary variable
\beq
  \label{u}
  u=K(\vartheta)\,,
\eeq
let 
\beq
  \label{beta}
  \beta(r)=h(\gamma(r)), \, \ r\in \Ar, 
\eeq
and set
\beq
  \label{h_gamma}
  u_{\Gamma}:=K(\vartheta_{\Gamma})\,, \, \ h_\Gamma:=\beta(u_{\Gamma}) \quad \hbox{ in } \Gamma_\ell \times [0,T]\, , 
\eeq
we arrive at the following formulation in terms of $u$:
\beq
  \label{system3_bis_2}
  \begin{split}
  \int_{\Omega}{\frac{\partial\gamma(u)}{\partial t}(t)w\,\pier{dv}}+\int_{\Omega}{\nabla u(t)\cdot\nabla w\,\pier{dv}}+\int_{\Gamma_{\pier{\ell}}}{\beta(u(t))w\,ds}\\
  =\int_{\Omega}{\sigma(\gamma(u)(t) ) |\nabla\phi(t) |^2 w\,\pier{dv}}+\int_{\Gamma_{\pier{\ell}}}{h_{\Gamma}(t)w\,ds} \\ \forall\, w\in W^{1,\infty} (\Omega)\,,\; \text{for a.e. } t\in[0,T]\,.
  \end{split}
\eeq
Please note that, here and in the sequel,  we use the notation ``$\pier{dv}$'' 
(in place of the heavy $dxdz$) for the volume element in the integrals over $\Omega$ and 
``$ds$'' for the surface element in the integrals on the boundaries of $\Omega$. We also observe that in \eqref{system3_bis_2} the test functions $w$ are taken in the smoother space $ W^{1,\infty} (\Omega)$ in order to try to give a meaning to all the integrals, in particular to the first one on the right hand side. Actually, we can set a larger space for the test functions, as it is pointed out in the next statement. 

Hence, we are ready to present the main result of the paper, which ensures that the problem we are dealing with actually has a solution $(V,\phi,\vartheta)$: moreover, it specifies which spaces are considered and in which sense the solution is intended.

\begin{thm}
  \label{theorem}
  Assume that \eqref{bound_sig_k}--\eqref{gamma}, \eqref{beta}--\eqref{h_gamma} and 
  \begin{gather}
    \label{data1}
   f\in L^1(0,T) ,\quad \vartheta_{\Gamma}\in L^2(\Sigma_\ell ),\\
    \label{data2}
   h,\,  \beta: \Ar\rarr\Ar \text{ continuous and increasing}\,, \quad h(0)= \beta(0)=0
  \end{gather}
  hold. Moreover, there are two constants $\vartheta_*>0$ and  $C_\beta >0$ such that
\begin{gather}
    \label{infess}
    \vartheta_0 \geq  \vartheta_* \, \hbox{ a.e. in } \Omega, \quad 
    \vartheta_{\Gamma}\geq\vartheta_*  \quad \text{a.e. in } \Sigma_\ell  \,,
\\
    \label{beta_lin}
    \left|\beta (r)\right|\leq C_\beta \bigl(1+\min\left\{\widehat{\beta}(r),\left|r\right|\right\}\bigr) \quad \forall\, r\in\Ar\,,
\end{gather}
  where $\widehat{\beta}$ is the (convex and $C^1$) primitive of $\beta$ such that $\widehat{\beta}(0)=0$.
  Then, there exists a quintuplet $(V,\phi, \psi, \vartheta, u)$ such that
  \begin{gather}
    \label{sol}
    V\in C^{0,1}\bigl([0,T]\bigr), \quad \phi\in L^\infty\bigl(0,T;\H1\bigr), 
    \quad \psi\in L^\infty\bigl(0,T;\gH1\bigr), \\
    \label{sol'}
    \vartheta\in W^{1,p}\bigl(0,T;W^{1,q}(\Omega)'\bigr), \quad \vartheta, \, u \in  L^p(0,T;W^{1,p}(\Omega))
  \end{gather}
  for some conjugate exponents $p\in (1, 5/4)$ and $q\in (5, +\infty) $, $\frac1p +\frac1q=1$,  and satisfying
  \beq
    \begin{split}
      \lambda_1V'(t)+\lambda_2V(t)+\lambda_3\int_0^t{V(r)\,dr}=
      \lambda_1V'_0+\lambda_2V_0-\int_B{\sigma(\vartheta(t))\frac{\partial\phi}{\partial z}(t)\,dx}\\
      \label{1}
      + \int_B{\sigma(\vartheta_0)\frac{\partial\phi_0}{\partial z}\,dx}+\int_0^t{f(r)\,dr}\,,
\end{split}
\eeq
\beq
    \label{2}
    \int_{\Omega}{\sigma(\vartheta(t))\nabla\psi(t)\cdot\nabla w\,\pier{dv}}+
    \int_{\Omega}{\sigma(\vartheta(t))\frac{V(t)}{\pier{\ell}}\frac{\partial w}{\partial z}\,\pier{dv}}=0 \quad
    \forall\, w\in \gH1\,, 
\eeq
\beq
\begin{split}
   &\int_{\Omega}{\frac{\partial\vartheta}{\partial t}(t)w\,\pier{dv}}+\int_{\Omega}{\nabla u(t)\cdot\nabla w\,\pier{dv}}+\int_{\Gamma_{\pier{\ell}}}{\beta(u(t))w\,ds}\\
    \label{3}
    &{}=\int_{\Omega}{\sigma(\vartheta(t))\left|\nabla\phi(t)\right|^2w\,\pier{dv}}+\int_{\Gamma_{\pier{\ell}}}{h_{\Gamma}(t)w\,ds} \quad\forall\, w\in W^{1,q}(\Omega)\, ,
    \end{split}
    \eeq
for almost every $t\in[0,T]$ and      
    \begin{gather}
    \label{psi_u}
    \phi(x,z,t) =\psi(x,z,t) + \frac{z}{\pier{\ell}}V(t)  
    \quad \hbox{for a.e. } (x,z,t) \in Q\, ,
\\
\label{pier5}
     \vartheta = \gamma (u) \quad \hbox{a.e. in } \, Q, 
 \\[0.1cm]
    \label{init}
    V(0)=V_0\,, \quad \vartheta(0)=\vartheta_0\,.
    \end{gather}
\end{thm}
\begin{rmk}
  We observe that the initial condition for $\vartheta$ in \eqref{init} makes sense at least in 
  $W^{1,q}(\Omega)'$ due to \eqref{sol'}, \eqref{pier3} and the inclusion $L^2(\Omega )\subset W^{1,q}(\Omega)'  $. Moreover, as $q>3$, let us point out that, owing to the Sobolev embedding results, not only  $L^2(\Omega ) $ but $L^1(\Omega )$ is continuosly embedded into $W^{1,q}(\Omega)'  $.
\end{rmk}
\begin{rmk}
  Please note that condition \eqref{beta_lin} is a priori very restrictive as a growth condition for $\beta$.  However, hypothesis \eqref{beta_lin} can be considered acceptable if we bear in
  mind our application: the input of the function $\beta$ is $u$, i.e. the temperature (modified through the Lipshitz-operator $K$), so that \eqref{beta_lin} trivially holds true if we confine our system to the case in which the temperature assumes values in a bounded 
  interval. This is cleary acceptable from the physical interpretation: for example, we can deal with 
  all situations in which the temperature does not exceed an arbitrary high level. Thus, this apparently restrictive mathematical hypothesis does not affect
  the interpretation of the problem.
\end{rmk}
\begin{rmk}
Assumption \eqref{infess} establishes a positive bound from below for the initial and some boundary values of the absolute temperature, which is completely reasonable and will help us to infer the same bound for the variable $\vartheta$. 
About condition \eqref{beta_lin}, we would like to comment also on the 
right hand side of the inequality:  by requiring that the growth of $\beta$ 
is controlled by a linearly growing function, 
we can find some useful estimate for suitable norms of $\frac{\partial \vartheta}{\partial t}$ in order to pass 
to the limit; furthermore, the application of the result of \cite{scarpa} during the delay argument
needs that $| \beta | $ is somehow controlled by $\widehat{\beta}$ plus a constant.
\end{rmk}


\section{Approximating the problem}
\setcounter{equation}{0}
\label{approx}

As we have anticipated, the idea is to prove Theorem \ref{theorem} using a delay argument, that we will describe in this section.
Firstly, let us introduce the delay parameter $\tau \in (0,T) $ and define the 
value of $\vartheta$ also for negative times in the following way:
\beq
\label{primadi0}
  \vartheta(t):=\vartheta_0 \quad\text{for all } t<0\,.
\eeq
Let now focus on the interval $[0,\tau]$ and consider the system 
\eqref{1}--\eqref{3}, in which we introduce 
some delay terms involving $\vartheta(t-\tau )$ and a 
truncation in $|\nabla\phi(t)|^2$ : more precisely, we use
\[
  \vartheta(t-\tau)  \,\text{ instead of } \, \vartheta(t)\, \text{ in \eqref{1}, \eqref{2} and the right hand side of \eqref{3}}\,,
\]
\[
  T_\tau\bigl(\left|\nabla\phi(t)\right|^2\bigr) \, \text{ instead of } \,  \left|\nabla\phi(t)\right|^2\,\text{ in \eqref{3}}\, ,
\]
where $T_\tau$ is the truncation operator
\beq
  \label{trunc}
  T_\tau(r)=
  \begin{cases}
      - {1}/{\tau} \quad&\text{if \ $r<-{1}/{\tau}$}\\
    r \quad&\text{if \ $\left|r\right|\leq{1}/{\tau}$}\\
    {1}/{\tau} \quad&\text{if \ $r>{1}/{\tau}$}
  \end{cases}
  \,,\qquad r\in\Ar\,.
\eeq
In other words, we consider the approximated system
\beq
 \begin{split}
      \lambda_1V'(t)+\lambda_2V(t)+\lambda_3\int_0^t{V(r)\,dr}=
      \lambda_1V'_0+\lambda_2V_0-\int_B{\sigma(\vartheta(t-\tau))\frac{\partial\phi}{\partial z}(t)\,dx}\\
      \label{1app}
      + \int_B{\sigma(\vartheta_0)\frac{\partial\phi_0}{\partial z}\,dx}+\int_0^t{f(r)\,dr}\,,
\end{split}
\eeq
\beq      
    \label{2app}
    \int_{\Omega}{\sigma(\vartheta(t-\tau))\nabla\psi(t)\cdot\nabla w\,\pier{dv}}+
    \int_{\Omega}{\sigma(\vartheta(t-\tau))\frac{V(t)}{\pier{\ell}}\frac{\partial w}{\partial z}\,\pier{dv}}=0 \quad
    \forall\, w\in \gH1\,,
\eeq
\beq
\begin{split}
    &\int_{\Omega}{\frac{\partial\vartheta}{\partial t}(t)w\,\pier{dv}}+\int_{\Omega}{\nabla u(t)\cdot\nabla w\,\pier{dv}}+\int_{\Gamma_{\pier{\ell}}}{\beta(u(t))w\,ds}\\
    \label{3app}
    &=\int_{\Omega}{\sigma(\vartheta(t-\tau))T_\tau\bigl(\left|\nabla\phi(t)\right|^2\bigr)w\,\pier{dv}}
    +\int_{\Gamma_{\pier{\ell}}}{h_{\Gamma}(t)w\,ds} \quad\forall\, w\in W^{1,q}(\Omega)\, ,
\end{split}
\eeq
\beq
    \label{psi_u_app}
    \phi(x,z,t) =\psi(x,z,t) + \frac{z}{\pier{\ell}}V(t)  
    \quad \hbox{for a.e. } (x,z) \in \Omega\, ,
\eeq
for almost every $t\in[0,\tau]$ and      
    \begin{gather}
\label{pier5_app}
     \vartheta = \gamma (u) \quad \hbox{a.e. in } \, \Omega \times [0,\tau], 
 \\[0.1cm]
    \label{init_app}
    V(0)=V_0\,, \quad \vartheta(0)=\vartheta_0\,.
    \end{gather}
Thus, we look for a solution on the interval $[0,\tau]$.

The idea at this level is to use a fixed point argument: let us explain in a first intuitive approach how we will proceed. We fix a suitable $\overline{V}$ in equation \eqref{2app}, which becomes 
in this way an explicit elliptic equation; we solve it and recover the solution $\psi$. From \eqref{psi_u_app}, with $\overline{V}$ in place of $V$, we are able to find the corresponding $\phi$
and substitute it in equation \eqref{1app}: thus, this one becomes a linear ordinary differential equation with known terms on the right hand side. Then, if we consider the Cauchy problem 
given by \eqref{1app} and the initial condition in  \eqref{init_app}, we obtain a unique
solution $V$ in $[0,\tau]$. Now, by considering the application mapping  $\overline{V}$ 
into $V$, we will deduce some contraction estimates provided $\tau$ is sufficiently small. Consequently, if the interval $[0,\tau]$ is suitably chosen, then there is a unique fixed point $V$ that solves the system involving \eqref{1app}, \eqref{2app}. Of course, the corresponding $\psi$ and $\phi$ can also be determined  
easily from  \eqref{2app} and \eqref{psi_u_app}. 
Once we have found the triplet $(V,\phi,\psi)$, we can check that the result contained in \cite{scarpa} applies to equation~\eqref{3app}, complemented by the relation \eqref{pier5_app} and the initial condition in \eqref{init_app}; moreover, since the terms on the right hand side of \eqref{3app} are smoother, in particular $\sigma(\vartheta(\, \cdot\,  -\tau))T_\tau\bigl(\left|\nabla\phi \right|^2\bigr)$ is uniformly bounded, there will be a solution pair $(\vartheta, u)$ 
with $\vartheta $ continuous from $[0,\tau]$ to $L^2(\Omega)$. 

At this point, the idea is to focus on the interval $[\tau,2\tau]$ and repeat the same argument used in $[0,\tau]$ for the new interval, with the new initial values
$V(\tau), \, \vartheta(\tau)$ and with $\vartheta (\, \cdot\, - \tau)$ specified by the solution  component $\vartheta$ found at the previous step. Note that in this case \eqref{1app} should be rewritten as
\beq
 \begin{split}
      \lambda_1V'(t)+\lambda_2V(t)+\lambda_3\int_\tau^t{V(r)\,dr}=
      \lambda_1V'_0+\lambda_2V_0 - \lambda_3 \int_0^\tau {V(r)\,dr}
      \\
     \nonumber
      -\int_B{\sigma(\vartheta(t-\tau))\frac{\partial\phi}{\partial z}(t)\,dx}+ \int_B{\sigma(\vartheta_0)\frac{\partial\phi_0}{\partial z}\,dx}+\int_0^t{f(r)\,dr}\,,
\end{split}
\eeq
for $t\in [\tau, 2\tau]$, and the term $ - \lambda_3 \int_0^\tau {V(r)\,dr}$ on the right hand side is now a datum 
since the solution $V$ has been already found from the previous iteration.
Then, one should  
continue in this way for every interval $[n\tau,(n+1)\tau]$, with $n$ going from $2$ to 
a value $N$ with $(N+1) \tau \geq T$. Hence, if it is possible to correctly {\em paste} 
together the solutions at each interval, then we will find a solution 
of the approximated problem in the whole interval $[0,T]$. 
The most important thing is that we have to check that the contraction estimates 
used in the fixed point argument at each iteration do not depend on the specific 
interval we consider, or, in other words, that there exists some $\tau >0$, small enough and independent of the sub-intervals, such that the contraction estimates hold.

Let us now present the proof of the existence of a solution for the \pier{approximating} problem. The key of the argument is the following lemma, which shows that some contraction estimates hold for a suitable value of $\tau$ on each sub-interval.
\begin{lem}
  \label{lemma}
  Let $\rho \in [0,T]$ and $\tau>0 $ such that $\rho +\tau \leq T$.
Given   $\overline{V}\in C^0([0,\rho+\tau])$ and
$ \overline{\vartheta} \in C^0\bigl([\rho,\rho+\tau ];\L2\bigr)$,
  there exists a unique triplet $(V,\phi, \psi)$ satisfying 
   \begin{gather}
    \label{V_psi}
    V\in C^{1}\bigl([\rho,\rho+\tau]\bigr), \quad \phi\in C^0\bigl([\rho,\rho+\tau];\H1\bigr), \quad \psi\in C^0\bigl([\rho,\rho+\tau];\gH1\bigr)
  \end{gather}
and solving the problem
\beq
 \begin{split}
      \lambda_1V'(t)+\lambda_2V(t)+\lambda_3\int_{\rho}^t{V(r)\,dr}=
      \lambda_1V'_0+\lambda_2V_0 - \lambda_3\int_0^{\rho}{\overline{V}(r)\,dr}
     \\
      \label{1lem}
       -\int_B{\sigma(\overline{\vartheta}(t))\frac{\partial\phi}{\partial z}(t)\,dx} 
       + \int_B{\sigma(\vartheta_0)\frac{\partial\phi_0}{\partial z}\,dx}+\int_0^t{f(r)\,dr}\,,
\end{split} 
\eeq      
\beq
    \label{2lem}
    \int_{\Omega}{\sigma(\overline{\vartheta}(t))\nabla\psi(t)\cdot\nabla w\,\pier{dv}}+
    \int_{\Omega}{\sigma(\overline{\vartheta}(t))\frac{\overline{V}(t)}{\pier{\ell}}\frac{\partial w}{\partial z}\,\pier{dv}}=0 \quad
    \forall\, w\in \gH1 \, ,
\eeq   
\beq
\label{psi_lem}
    \phi(x,z,t) =\psi(x,z,t) + \frac{z}{\pier{\ell}}\overline{V} (t)  
    \quad \hbox{for a.e. } (x,z) \in \Omega\, ,
\eeq
\noindent for all $t\in[\rho,\rho+\tau] $, and      
    \begin{gather}
    \label{init_lem}
    V(\rho)=\overline{V} (\rho) .
    \end{gather} 
  Furthermore, let 
 $\overline{V}_1, \overline{V}_2\in C^0([0,\rho+\tau])$ 
 satisfy $\overline{V}_1(t)=\overline{V}_2(t)$ for all $t\in[0,\rho]$
  and let $V_1, V_2$ denote the corresponding solution 
  components of the problem  \eqref{V_psi}--\eqref{init_lem}. Then, 
  the following estimate holds:
  \begin{gather}
    \label{contr2}
    \l|V_1-V_2\r|_{L^\infty(\rho,\rho+\tau)}\leq
    \sqrt{\frac{\tau}{\lambda_1\lambda_2}}\, \frac{2\sigma^*\left|B\right|}{\pier{\ell}}
    \l|\overline{V}_1-\overline{V}_2\r|_{L^\infty(\rho,\rho+\tau)}\, ,
  \end{gather}
 where $|B|$ denotes the bidimensional measure of the set $B$ in \eqref{omega}.
\end{lem}
\begin{proof}
  For all $t\in[\rho,\rho+\tau]$, it natural to introduce 
  the bilinear form $a_t:\gH1\times\gH1\rarr\Ar$ as
  \beq
    \label{bilin}
    a_t(w_1,w_2):=\int_\Omega{\sigma(\overline{\vartheta}(t))\nabla w_1\cdot\nabla w_2\,\pier{dv}}\,, \quad  w_1, w_2 \in \gH1.
  \eeq
Using \eqref{bound_sig_k} and the Poincar\'e inequality, it is easy to check that $a_t$ is continuous and coercive on $\gH1$ . Moreover, we see that the linear functional 
  \[
    w\mapsto -\int_\Omega{\sigma(\overline{\vartheta}(t))\frac{\overline{V}(t)}{\pier{\ell}}\frac{\partial w}{\partial z}\,\pier{dv}}
  \]
  is continuous on $\gH1$ for all $t\in[\rho,\rho+\tau]$:
  hence, the Lax-Milgram lemma ensures that for all $t\in[\rho,\rho+\tau]$ there exists a unique $\psi(t)\in\gH1$
  such that
  \[
    \int_{\Omega}{\sigma(\overline{\vartheta}(t))\nabla\psi(t)\cdot\nabla w\,\pier{dv}}+
    \int_{\Omega}{\sigma(\overline{\vartheta}(t))\frac{\overline{V}(t)}{\pier{\ell}}\frac{\partial w}{\partial z}\,\pier{dv}}=0 \quad
    \forall\, w\in \gH1\,.
  \]
  Moreover, testing the previous expression by $w=\psi(t)\in\gH1$ and using the Young inequality, it easily follows that
  \beq
    \label{est_psi_V}
    \int_\Omega{\sigma(\overline{\vartheta}(t))\left|\nabla \psi(t)\right|^2\,\pier{dv}}\leq
    \int_\Omega{\sigma(\overline{\vartheta}(t))\frac{\overline{V}^2(t)}{\ell^2}\,\pier{dv}}\,.
  \eeq
Due to the continuity properties   $\overline{\vartheta}\in C^0\bigl([\rho,\rho+\tau];\L2\bigr)$, $\overline{V}\in C^0\bigl([\rho,\rho+\tau]\bigr)$ and to 
\eqref{bound_sig_k}, we also infer that
  \[
    \psi\in C^0\bigl([\rho,\rho+\tau]; \gH1\bigr)\,.
  \]
Next, let us find $\phi\in C^0([\rho,\rho+\tau]; \H1)$ through \eqref{psi_lem}. 
At this point, we note that the same argument leading to~\eqref{no_zdep'}
ensures that the term
  \[
    \int_B{\sigma(\overline{\vartheta}(t))\frac{\partial\phi}{\partial z}(t)\,dx}
  \]
does not depend on $z$, and equation~\eqref{1lem} makes sense. Then, equation \eqref{1lem}
  becomes an ordinary differential equation in $[\rho,\rho+\tau]$ with appropriate initial conditions \eqref{init_lem}.  By virtue of \eqref{bound_sig_k}--\eqref{pier4}, \eqref{data1}
  and  $\overline{V}\in C^0\bigl([0,\rho]\bigr)$, the right hand side of \eqref{1lem} is a continuous function in  $[\rho,\rho+\tau]$. Hence, it is straightforward to see that 
there exists a unique solution
  \[
    V\in C^1\bigl([\rho,\rho+\tau]\bigr)
  \]
  of the Cauchy problem expressed by \eqref{1lem} and \eqref{init_lem}. Thus, we have proved the first part of the lemma. Let us focus now on the 
  contraction estimates. Given $\overline{V}_1, \overline{V}_2\in C^0([0,\rho+\tau])$ such that $\overline{V}_1=\overline{V}_2$ in $[0,\rho]$,
  let $(V_i, \psi_i, \phi_i)$, $i=1,2$, be the corresponding solutions of problem  \eqref{V_psi}--\eqref{init_lem}.  Then, if we consider \eqref{1lem} and 
  take the difference, testing by $V_1(t)-V_2(t)$ leads to
  \[
    \begin{split}
    \frac{\lambda_1}{2}\frac{d}{dt}&\left|V_1(t)-V_2(t)\right|^2+
    \lambda_2 \left|V_1(t)-V_2(t)\right|^2+
    \frac{\lambda_3}{2}\frac{d}{dt}\bigl(\int_{\rho}^t{\bigl(V_1(r)-V_2(r)\bigr)\,dr}\bigr)^2\\
    &=\bigl(V_1(t)-V_2(t)\bigr)\cdot\int_B{\sigma\bigl(\overline{\vartheta}(t)\bigr)\frac{\partial}{\partial z}\bigl(\phi_2(t)-\phi_1(t)\bigr)\,dx}\\
    &\leq\frac{\lambda_2}{2}\left|V_1(t)-V_2(t)\right|^2+
    \frac{1}{2\lambda_2}\bigl(\int_B{\sigma\bigl(\overline{\vartheta}(t)\bigr)\frac{\partial}{\partial z}\bigl(\phi_2(t)-\phi_1(t)\bigr)\,dx}\bigr)^2. 
    \end{split}
  \]
Hence, integrating on $(\rho,t)$ we deduce that
  \[
    \begin{split}
    \frac{\lambda_1}{2}\left|V_1(t)-V_2(t)\right|^2&+
    \frac{\lambda_2}{2}\int_{\rho}^t{\left|V_1(r)-V_2(r)\right|^2\,dr}+
    \frac{\lambda_3}{2}\bigl(\int_{\rho}^t{\bigl(V_1(r)-V_2(r)\bigr)\,dr}\bigr)^2\\
    &\leq\frac{1}{2\lambda_2}\int_{\rho}^t\bigl(\int_B{\sigma\bigl(\overline{\vartheta}(r)\bigr)\left|\nabla(\phi_1(r)-\phi_2(r))\right|\,dx}\bigr)^2\,dr
    \end{split}
  \]
and the H\"older inequality and condition~\eqref{bound_sig_k} allow us to infer 
  \[
    \begin{split}
    \lambda_1\left|V_1(t)-V_2(t)\right|^2 &\leq
    \frac{1}{\lambda_2}\int_{\rho}^t \bigl(\int_B{\sigma\bigl(\overline{\vartheta}(r)\bigr)\,dx}\bigr)\cdot
    \int_B{\sigma\bigl(\overline{\vartheta}(r)\bigr)\left|\nabla(\phi_1(r)-\phi_2(r))\right|^2\,dx\,dr}\\
    &\leq\frac{\sigma^*\left|B\right|}{\lambda_2}\int_{\rho}^t\int_B{\sigma\bigl(\overline{\vartheta}(r)\bigr)\left|\nabla(\phi_1(r)-\phi_2(r))\right|^2\,dx\,dr}
    \end{split}
  \]
    for all $t\in[\rho,\rho+\tau]$ . Now,  
taking  \eqref{psi_lem} and \eqref{est_psi_V} into account, 
 for all $t\in[\rho,\rho+\tau]$ we have
  \[
     \begin{split}
     \lambda_1\left|V_1(t)-V_2(t)\right|^2 &\leq
     2\, \frac{\sigma^*\left|B\right|}{\lambda_2}\int_{\rho}^t\int_B{\sigma
     \bigl(\overline{\vartheta}(r)\bigr)\left|\nabla(\psi_1(r)-\psi_2(r))\right|^2\,dx\,dr}\\
     &\quad{}+2 \, \frac{\sigma^*\left|B\right|}{\lambda_2}
     \int_{\rho}^t\int_B{\sigma\bigl(\overline{\vartheta}(r)\bigr)
     \frac{\left|\overline{V}_1(r)-\overline{V}_2(r)\right|^2}{\ell^2}\,dx\,dr}\\
     &\leq\frac{4}{\lambda_2}\bigl(\frac{\sigma^*\left|B\right|}{\pier{\ell}}\bigr)^{\! 2}
     \int_{\rho}^{\rho+\tau}{\left|\overline{V}_1(r)-\overline{V}_2(r)\right|^2\,dr}
     \\
    &\leq\frac{\tau}{\lambda_2}\bigl(\frac{2 \sigma^*\left|B\right|}{\pier{\ell}}\bigr)^{\!2}  
    \left\|\overline{V}_1-\overline{V}_2 \right\|_{L^\infty ( \rho,\rho+\tau) }^2 \, .
     \end{split}
  \]
At this point, passing to the square roots it is straightforward to derive the final estimate \eqref{contr2}.
\end{proof}

Another auxiliary result is formulated in the next lemma. 

\begin{lem}
  \label{lempier}
  Let $\rho \in [0,T]$ and $\tau>0 $ such that $\rho +\tau \leq T$.
  For all  
$$ 
  \overline{\vartheta} \in C^0\bigl([\rho,\rho+\tau ];\L2\bigr), \quad
  \overline{\phi}\in C^0\bigl([\rho,\rho+\tau ];\H1\bigr), \quad 
  \Theta_{\rho}\in \L2
$$
there exists a pair $(\vartheta, u)$ satisfying 
\beq
  \label{pier6} 
    \vartheta  \in H^1\bigl(\rho,\rho+\tau;\H1'\bigr), \quad \vartheta , \, u \in  C^0\bigl([\rho,\rho+\tau];\L2\bigr)\cap L^2\bigl(\rho,\rho+\tau;\H1\bigr)
\eeq
and solving the problem
\beq
\begin{split}
    &\int_{\Omega}{\frac{\partial\vartheta}{\partial t}(t)w\,\pier{dv}}+\int_{\Omega}{\nabla u(t)\cdot\nabla w\,\pier{dv}}+\int_{\Gamma_{\pier{\ell}}}{\beta(u(t))w\,ds}\\
    \label{pier7}
    &=\int_{\Omega}{\sigma(\overline{\vartheta}(t))\, T_\tau\bigl(\left|\nabla\overline{\phi}(t)\right|^2\bigr)w\,\pier{dv}}
    +\int_{\Gamma_{\pier{\ell}}}{h_{\Gamma}(t)w\,ds} \quad\forall\, w\in H^1(\Omega),
\end{split}
\eeq
 for a.e. $ t\in [\rho,\rho+\tau],$
    \begin{gather}
\label{pier8}
     \vartheta = \gamma (u) \quad \hbox{a.e. in } \, \Omega\times  [\rho,\rho+\tau],
 \\[0.1cm]
    \label{pier9}
 \vartheta(\rho)=\Theta_{\rho} \,.
    \end{gather}
\end{lem}
\begin{proof}
Note that the term $\sigma(\overline{\vartheta})\, T_\tau\bigl(\left|\nabla\overline{\phi}
\right|^2\bigr)$ appearing in the right hand side of \eqref{pier7} is in $L^\infty ( \Omega\times  [\rho,\rho+\tau]) $ due to \eqref{bound_sig_k} and \eqref{trunc},  whence  
$$ \sigma(\overline{\vartheta})\, T_\tau\bigl(\left|\nabla\overline{\phi}\right|^2\bigr) \in L^2\bigl(\rho,\rho+\tau;\L2\bigr) \cap L^1\bigl(\rho,\rho+\tau;L^\infty (\Omega)\bigr)$$
for all $\tau >0$. Let us recall also the assumptions \eqref{gamma}, 
\eqref{beta}, \eqref{h_gamma}, \eqref{data1}, \eqref{data2}, \eqref{beta_lin} for the graphs (actually, functions) $\gamma$ and $\beta$ and for the boundary datum $h_\Gamma \in L^2\bigl(\rho,\rho+\tau;L^2(\gamma_\ell)\bigr) $. Then, as $\Theta_{\rho} \in \L2$, 
we can apply \cite[Thm.~2.3]{scarpa} (see also \cite[Rem.~2.3]{scarpa}) and infer the existence of a pair  $(\vartheta, u)$ with the regularity specified by \eqref{pier6}.
\end{proof}

We are now ready to describe how to obtain a solution of the \pier{approximating} problem.
The idea is to build a solution $(V_\tau, \phi_\tau, \psi_\tau, \vartheta_\tau, u_\tau)$ step by step. 

Let us define $\vartheta_\tau (t) $ for $t \leq 0$ as in  \eqref{primadi0} and 
focus on the first interval $[0,\tau]$; consider $\overline{V}\in C^0([0,\tau])$
with $\overline{V} (0)= V_0.$  Then, Lemma~\ref{lemma} with the choice $\overline{\vartheta}(t) = \vartheta_\tau (t- \tau)= \vartheta_0$, $t\in [0,\tau]$,  tells us that there exists a unique triplet 
\[
  \bigl(V,\phi,\psi\bigr)\in C^1\bigl([0,\tau]\bigr)\times C^0\bigl([0,\tau];\H1\bigr) 
  \times C^0\bigl([0,\tau];\gH1\bigr)
\]
which solves the problem  \eqref{1lem}--\eqref{init_lem}. Hence, it is natural to introduce the
operator
\beq
  \label{contr_op}
  \Lambda_0: C^0\bigl([0,\tau]\bigr)\rarr C^1\bigl([0,\tau]\bigr), \quad \Lambda_0(\overline{V})=V
\eeq
and observe that, thanks to \eqref{contr2}, we can fix $\tau$ such that
\beq
\label{pier11}
0< \tau < \tau^*:= \frac{\lambda_1\lambda_2 \, \ell^2}{( 2\sigma^*\left|B\right|)^2} ,
\eeq 
that is, in order that \eqref{contr_op} be a contraction mapping.
Hence, there exists a unique
\[
  V^{(0)}\in C^1\bigl([0,\tau]\bigr)
\]
which is a fixed point for $\Lambda_0$, i.e. a solution, along with the related  
\[
\phi^{(0)} \in C^0\bigl([0,\tau];\H1\bigr)
 \hbox{ and }\,   
  \psi^{(0)}\in C^0\bigl([0,\tau];\gH1\bigr) ,
\]
of the problem expressed by \eqref{1app}, \eqref{2app}, \eqref{psi_u_app}, and the initial condition  $V(0)=V_0$; actually the triplet $ (V^{(0)}, 
\phi^{(0)} ,  \psi^{(0)})$ is the unique solution of this problem in  $[0,\tau]$. 

Next,
we choose $\rho=0$, $\overline{\vartheta}(t) = \vartheta_\tau (t- \tau)= \vartheta_0$ and 
$\overline{\phi}(t) = \phi^{(0)} (t)$ for $t\in [0,\tau]$,  besides the initial value $\Theta_0= \vartheta_0$,
in Lemma~\ref{lempier} and find a pair $(\vartheta, u) = (\vartheta^{(0)} , u^{(0)})$ with 
\[
  \vartheta^{(0)}\in H^1\bigl(0,\tau;\H1'\bigr), \quad 
  \vartheta^{(0)} , u^{(0)} \in 
 C^0\bigl([0,\tau];\L2\bigr)\cap L^2\bigl(0,\tau;\H1\bigr)
\]
solving \eqref{pier7}--\eqref{pier9} in $[0,\tau]$.  Hence, the quintuplet $ (V^{(0)}, 
\phi^{(0)} ,  \psi^{(0)},  \vartheta^{(0)} , u^{(0)} )$ yields a solution to~\eqref{1app}--\eqref{init_app} in $[0,\tau]$ and it is natural to define
\[
(V_\tau, \phi_\tau, \psi_\tau, \vartheta_\tau, u_\tau):= (V^{(0)}, 
\phi^{(0)} ,  \psi^{(0)},  \vartheta^{(0)} , u^{(0)} ) \quad\text{in } [0,\tau]\,.
\]

Let us construct now the solution of \pier{approximating} problem on the sub-interval $[\tau, 2\tau]$. Here, the delay terms are associated to the functions found in the previous
passage. We consider thus the problem
\beq
\begin{split}
      \lambda_1V'(t)+\lambda_2V(t)+\lambda_3\int_\tau^t{V(r)\,dr}=
      \lambda_1V'_0+\lambda_2V_0 - \lambda_3\int_0^\tau{V_\tau (r)\,dr}
      \\
      -\int_B{\sigma(\vartheta_{\tau}(t-\tau))\frac{\partial\phi}{\partial z}(t)\,dx}
      \label{1app'}
      + \int_B{\sigma(\vartheta_0)\frac{\partial\phi_0}{\partial z}\,dx}+\int_0^t{f(r)\,dr},
\end{split}
\eeq
\beq
    \int_{\Omega}{\sigma(\vartheta_{\tau}(t-\tau))\nabla\psi(t)\cdot\nabla w\,\pier{dv}}+
    \int_{\Omega}{\sigma(\vartheta_{\tau}(t-\tau))\frac{V(t)}{\pier{\ell}}\frac{\partial w}{\partial z}\,\pier{dv}}=0\quad
    \label{2app'}
    \forall \,  w\in \gH1\,,\\
\eeq
\beq
\begin{split}
    &\int_{\Omega}{\frac{\partial\vartheta}{\partial t}(t)w\,\pier{dv}}+\int_{\Omega}{\nabla u(t)\cdot\nabla w\,\pier{dv}}+\int_{\Gamma_{\pier{\ell}}}{\beta(u(t))w\,ds}\\
    \label{3app'}
    &=\int_{\Omega}{\sigma(\vartheta_{\tau}(t-\tau))T_\tau\bigl(\left|\nabla\phi(t)\right|^2\bigr)w\,\pier{dv}}
    +\int_{\Gamma_{\pier{\ell}}}{h_{\Gamma}(t)w\,ds} \quad\forall\, w\in H^{1}(\Omega)\,,\end{split}
    \eeq
  \beq
\label{psi_altra}
    \phi(x,z,t) =\psi(x,z,t) + \frac{z}{\pier{\ell}}V (t)  
    \quad \hbox{for a.e. } (x,z) \in \Omega\, ,
\eeq
\noindent for a.e. $t\in[\tau,2\tau] $, and       
\begin{gather}
\label{pier12}
     \vartheta = \gamma (u) \quad \hbox{a.e. in } \, \Omega\times  [\tau, 2\tau],
 \\[0.1cm]
    \label{init_app'}
    V(\tau)=V_\tau (\tau) \,, \quad \vartheta(\tau)=\vartheta_\tau (\tau).
\end{gather}
The same argument that we have used to prove the existence of a solution on the first interval can be easily repeated in this case in the same way.  Take $\rho=\tau$, $\overline{V} \in C^0([0,2\tau])$ such that $\overline{V} = V_\tau$ in $[0,\tau]$, and 
$\overline{\vartheta}(t) = \vartheta_\tau (t- \tau)$, $t\in [0,2\tau]$, in Lemma~\ref{lemma}. Find the corresponding  solution $(V,\phi,\psi)$ and observe that the obviously defined operator 
\beq
\nonumber
  \Lambda_1: C^0\bigl([\tau,2\tau]\bigr)\rarr C^1\bigl([\tau,2\tau]\bigr), \quad \Lambda_1\bigl(\overline{V}\bigr)=V\,,
\eeq
is a contraction mapping as well, due to \eqref{pier11} and \eqref{contr2}: this follows 
from the fact that the constant present in estimate~\eqref{contr2} does not depend 
on $\rho$.
Consequently, we are able to find a solution
\beq \nonumber
  V^{(1)}\in C^1\bigl([\tau,2\tau]\bigr),  \quad
\phi^{(1)} \in C^0\bigl([\tau,2\tau];\H1\bigr), \quad 
  \psi^{(1)}\in C^0\bigl([\tau,2\tau];\gH1\bigr) 
\eeq
of the problem \eqref{1app'}, \eqref{2app'}, \eqref{psi_altra}, and $V^{(1)} (\tau) = V_\tau (\tau).$ Then, we apply Lemma~\ref{lempier} with $\overline{\vartheta}(t) = \vartheta_\tau (t- \tau)$ and  $\overline{\phi}(t) = \phi^{(1)} (t)$ for $t\in [\tau, 2\tau]$,  choosing the initial value $\Theta_\tau = \vartheta_\tau (\tau)$. Thus, we obtain a solution 
\begin{gather} \nonumber
  \vartheta^{(1)}\in H^1\bigl(\tau,2\tau;\H1'\bigr), \quad   \vartheta^{(1)}, 
  u^{(1)}\in C^0\bigl([\tau,2\tau];\L2\bigr)\cap L^2\bigl(\tau,2\tau;\H1\bigr)
\end{gather}
of the problem \eqref{3app'}, \eqref{pier12}, and the second initial condition in \eqref{init_app'}.  At this point, we can extend our solution to the interval $[\tau, 2\tau]$ 
in the following way:
\[
(V_\tau, \phi_\tau, \psi_\tau, \vartheta_\tau, u_\tau):= (V^{(1)}, 
\phi^{(1)} ,  \psi^{(1)},  \vartheta^{(1)} , u^{(1)} ) \quad\text{in } (\tau, 2\tau]\,.
\]
Note that, by this definition, the {\em pasted} functions $\vartheta_\tau$ and $u_\tau $ remain continuous from $[0,2\tau] $ to $\L2$, and consequently also $\phi_\tau, \, \psi_\tau $
keep the continuity property in $[0,2\tau] $; then, by arguing on \eqref{1app'} and \eqref{init_app'}, we realize that  $V_\tau \in C^1\bigl([0,2\tau]\bigr)$.

The idea is to repeat this argument until we reach $T$. In general, when we solve the problem on $[n\tau, (n+1)\tau]$ we already have defined $(V_{\tau}, \psi_{\tau}, \vartheta_{\tau})$ on $[0,n\tau]$, so that it makes sense to consider the analogue of the problem \eqref{1app'}--\eqref{init_app'} in  $[n\tau, (n+1)\tau]$ by replacing $\tau$ with $n\tau$ 
in the extrema of the integrals in \eqref{1app'} and in the initial conditions \eqref{init_app'}.
Like before, the condition  \eqref{pier11} on $\tau$ ensures that the operator
\beq \nonumber
  \Lambda_n: C^0\bigl([n\tau, (n+1)\tau]\bigr)\rarr C^1\bigl([n\tau, (n+1)\tau]\bigr), \quad \Lambda_n\bigl(\overline{V}\bigr)=V
\eeq
is a contraction mapping, so that with the \luca{help} of both lemmas we can find a quintuplet 
$ (V^{(n)},  \phi^{(n)} ,  \psi^{(n)},  \vartheta^{(n)} , u^{(n)} )$ on $[n\tau, (n+1)\tau]$ 
and consequently update our solution by
\[
(V_\tau, \phi_\tau, \psi_\tau, \vartheta_\tau, u_\tau):= (V^{(n)}, 
\phi^{(n)} ,  \psi^{(n)},  \vartheta^{(n)} , u^{(n)} ) \quad\text{in } (n\tau, (n+1)\tau]\,.
\]
To reach exactly the final time $T$, it suffices to take $\tau = T/(n+1) $ with $n$ sufficiently large in order to satisfy \eqref{pier11}; otherwise, we solve the last iteration in the shorter interval $[n\tau, T] $ with $T < (n+1)\tau$.
It is now straightforward to deduce that $(V_\tau, \phi_\tau, \psi_\tau, \vartheta_\tau, u_\tau)$ is a solution for the \pier{approximating} problem \eqref{1app}--\eqref{init_app}.

Let us collect the result we have just proved in the following statement.
\begin{thm}
  \label{app_prob}
Under the assumptions of Theorem~\ref{theorem}, let $\tau \in (0, \tau^*)$. Then there  exists a quintuplet $ (V_\tau, \phi_\tau, \psi_\tau, \vartheta_\tau, u_\tau)$ such that
  \begin{gather}
    \label{sol1_app}
    V_{\tau}\in C^1\bigl([0,T]\bigr),\quad 
   \phi_{\tau}\in C^0\bigl([0,T]; \H1\bigr), \quad \psi_{\tau}\in C^0\bigl([0,T];\gH1\bigr), \\
    \label{sol3_app}
    \vartheta_{\tau}\in H^1\bigl(0,T;\H1'\bigr),  \quad  \vartheta_{\tau}, u_\tau  \in  
    C^0\bigl([0,T];\L2\bigr)\cap L^2\bigl(0,T;\H1\bigr)
  \end{gather}
and satisfying 
\beq
  \begin{split}
      \lambda_1V_{\tau}'(t)+\lambda_2V_{\tau}(t)+\lambda_3\int_0^t{V_{\tau}(r)\,dr}=
      \lambda_1V'_0+\lambda_2V_0-\int_B{\sigma(\vartheta_{\tau}(t-\tau))\frac{\partial\phi_{\tau}}{\partial z}(t)\,dx}\\
      \label{1'}
      + \int_B{\sigma(\vartheta_0)\frac{\partial\phi_0}{\partial z}\,dx}+\int_0^t{f(r)\,dr} \quad\text{for all $t\in[0,T]$}\,,
      \end{split}
\eeq
\beq
\begin{split}
    \int_{\Omega}{\sigma(\vartheta_{\tau}(t-\tau))\nabla\psi_{\tau}(t)\cdot\nabla w\,\pier{dv}}+
    \int_{\Omega}{\sigma(\vartheta_{\tau}(t-\tau))\frac{V_{\tau}(t)}{\pier{\ell}}\frac{\partial w}{\partial z}\,\pier{dv}}=0\\
    \label{2'}
    \forall \,  w\in \gH1\,, \quad\text{for all $\, t\in[0,T]$}\,,
\end{split}
\eeq
\beq
\begin{split}
    \int_{\Omega}{\frac{\partial\vartheta_{\tau}}{\partial t}(t)w\,\pier{dv}}+\int_{\Omega}{\nabla u_{\tau}(t)\cdot\nabla w\,\pier{dv}}
    +\int_{\Gamma_{\pier{\ell}}}{\beta(u_{\tau}(t))w\,ds}\\
  {}=\int_{\Omega}{\sigma(\vartheta_{\tau}(t-\tau))T_\tau\bigl(\left|\nabla\phi_{\tau}(t)\right|^2\bigr)w\,\pier{dv}}  +\int_{\Gamma_{\pier{\ell}}}{h_{\Gamma}(t)w\,ds}\\
    \label{3'}
    \forall \, w\in\H1\,, \quad \text{for a.e. $t\in[0,T]$}\,,
\end{split}
\eeq
\begin{gather}
    \label{psi_u'}
    \phi_{\tau}(x,z,t) =\psi_{\tau}(x,z,t) + \frac{z}{\pier{\ell}}V_{\tau}(t)  
    \quad \hbox{for a.e. } (x,z) \in \Omega, \hbox{ for all } t \in [0,T]\, ,\\
     \label{pier5'}
     \vartheta_{\tau} = \gamma (u_{\tau}) \quad \hbox{a.e. in } \, Q, 
     \\[0.1cm]
    \label{init'}
    V_{\tau}(0)=V_0\,, \quad \vartheta_{\tau}(0)=\vartheta_0\,.
    \end{gather}
\end{thm}


\section{Uniform estimates}
\setcounter{equation}{0}
\label{estimat}

In this section, we present some estimates, independent of $\tau$, on the solution of the \pier{approximating} problem \eqref{1'}--\eqref{init'} specified by Theorem~\ref{app_prob}. Then, we  aim to pass to the limit and recover a solution to the original problem \eqref{1}--\eqref{init}.

\subsection{The first estimate}
We would like to show that 
\beq
  \label{1est}
  \vartheta_{\tau}\geq\vartheta_{*} \quad\text{a.e. in } Q\,, \quad\text{for all } \tau \in (0,\tau^*),
\eeq
where the value $\vartheta^*>0$ comes out in assumption \eqref{infess}. 

The idea is to test equation \eqref{3'} by
\[
  w=-(\vartheta_{\tau}(t)-\vartheta_{*})^-\in\H1
\]
and handle the different terms separately.
Firstly, we can note that
\[
  -\int_{\Omega}{\frac{\partial\vartheta_{\tau}}{\partial t}(t)(\vartheta_{\tau}(t)-\vartheta_{*})^-\,\pier{dv}}=
  \frac{1}{2}\frac{d}{dt}\int_{\Omega}{\bigl((\vartheta_{\tau}(t)-\vartheta_{*})^-\bigr)^2\,\pier{dv}}\,;
\]
secondly, in view of \eqref{pier5'}, \eqref{bound_sig_k}, \eqref{K}, \eqref{gamma} we have that
\[
  \begin{split}
  &-\int_{\Omega}{\nabla{u_{\tau}(t)}\cdot\nabla(\vartheta_{\tau}(t)-\vartheta_{*})^-\,\pier{dv}}=
  -\int_{\Omega}{K'(\vartheta_{\tau}(t))\nabla{\vartheta_{\tau}(t)}\cdot\nabla(\vartheta_{\tau}(t)-\vartheta_{*})^-\,\pier{dv}}\\
  &=\left.
  \begin{cases}
    0 \quad &\text{if } \vartheta_{\tau}(t)\geq\vartheta_{*}\\
    \luca{\int_{\Omega}{k(\vartheta_{\tau}(t))\left|\nabla\vartheta_{\tau}(t)\right|^2\,\pier{dv}} } \quad &\text{if } \vartheta_{\tau}(t)< \vartheta_{*}
  \end{cases}
  \right\}  \geq0\,.
  \end{split}
\]
Moreover, let us verify that
\[
   -\int_{\Gamma_{\pier{\ell}}}{\bigl(\beta(K(\vartheta_{\tau}(t)))-\beta(K(\vartheta_{\Gamma}(t)))\bigr)(\vartheta_{\tau}(t)-\vartheta_{*})^-\,ds}\geq0\,.
\]
Indeed, if $\vartheta_{\tau}(t)\geq\vartheta_{*}$ this quantity is $0$, so let us consider the case $\vartheta_{\tau}(t)<\vartheta_{*}\, $:
as (cf.~\eqref{infess}) $\vartheta_{*}\leq\vartheta_{\Gamma}$ it turns out that $\vartheta_{\tau}(t)\leq\vartheta_{\Gamma}(t)$ almost everywhere in $\Gamma_\ell$. Hence,  \eqref{h_gamma} and the monotonicity of $\beta$ and $K$ ensure that
\[
  -\int_{\Gamma_{\pier{\ell}}}{\bigl(\beta(u_{\tau}(t))-h_{\Gamma}(t)\bigr)(\vartheta_{\tau}(t)-\vartheta_{*})^-\,ds}\geq0\,.
\]
Taking all these remarks into account,
from \eqref{3'} we deduce that
\[
  \begin{split}
  \frac{d}{dt}\int_{\Omega}{\frac{\bigl((\vartheta_{\tau}(t)-\vartheta_{*})^-\bigr)^2}{2}\,\pier{dv}}\leq
  &-\int_{\Omega}{\sigma(\vartheta_{\tau}(t-\tau))T_\tau\bigl(\left|\nabla\phi_{\tau}(t)\right|^2\bigr)(\vartheta_{\tau}(t)-\vartheta_{*})^-\,\pier{dv}}
  \leq0\,;
  \end{split}
\]
then, integrating with respect to $t$ and recalling \eqref{init'} and \eqref{infess}, we obtain
\[
  \int_{\Omega}{\bigl((\vartheta_{\tau}(t)-\vartheta_{*})^-\bigr)^2\,\pier{dv}}\leq
  \int_{\Omega}{\bigl((\vartheta_0-\vartheta_{*})^-\bigr)^2\,\pier{dv}}=0
  \quad\text{for a.e. } t\in[0,T]\,,
\]
which implies condition \eqref{1est}.

\subsection{The second estimate}
The aim of this subsection is to prove the existence of some positive constants $C_1, C_2$, independent of $\tau$,
such that the following estimates hold:
\begin{gather}
  \label{2est}
  \l|V_{\tau}\r|_{W^{1,\infty}(0,T)}\leq C_1 \quad\text{for all } \tau \in (0,\tau^*)\,,\\
  \label{3est}
  \l|\psi_{\tau}\r|_{L^\infty (0,T;\gH1)} +\l|\phi_{\tau}\r|_{L^\infty(0,T;\H1)}\leq C_2 \quad\text{for all } \tau \in (0,\tau^*)\,.
\end{gather}

We multiply equation \eqref{1'} by $V_{\tau}(t)$ and integrate. Using \eqref{bound_sig_k}, the independence of $z$ of the integrals over $B$ (cf.~\eqref{no_zdep'}) and the Young inequality, we deduce that
\[
  \begin{split}
  &\frac{\lambda_1}{2}\left|V_{\tau}(t)\right|^2+\lambda_2\int_0^t{V_{\tau}^2(r)\,dr}+\frac{\lambda_3}{2}\bigl(\int_0^t{V_{\tau}(r)\,dr}\bigr)^2 \\
  &=\frac{\lambda_1}{2}V_0^2
  +\int_0^t{\bigl(\lambda_1V_0'+\lambda_2 V_0+
  \int_B{\sigma(\vartheta_0)\frac{\partial\phi_0}{\partial z}\,dx}+\int_0^r{f(\rho)\,d\rho}\bigr)V_{\tau}(r)\,dr}\\
  &\quad{}-\int_0^t\bigl(\int_B{\sigma(\vartheta_{\tau}(r-\tau))\frac{\partial\phi_{\tau}}{\partial z}(r)\,dx}\bigr)V_{\tau}(r)\,dr\\
  &\leq
  \frac{\lambda_1}{2}V_0^2+\frac{1}{2\eps}\int_0^t{\bigl(\lambda_1V_0'+\lambda_2 V_0+
  \frac{1}{\pier{\ell}}\int_\Omega{\sigma(\vartheta_0)\frac{\partial\phi_0}{\partial z}\,\pier{dv}}+\int_0^r{f(\rho)\,d\rho}\bigr)^2dr}+\frac{\eps}{2}\int_0^t{V_{\tau}^2(r)\,dr}\\
  &\quad{}+\frac{1}{2\eps \ell}\int_0^t\int_\Omega{\sigma(\vartheta_{\tau}(r-\tau))\left|\frac{\partial\phi_{\tau}}{\partial z}(r)\right|^2\pier{dv}dr}+
  \frac{\eps}{2\ell}\int_0^t\int_\Omega{\sigma(\vartheta_{\tau}(r-\tau))V_{\tau}^2(r)\,\pier{dv}dr}\\
  &\leq\left[\frac{\lambda_1}{2}V_0^2+\frac{T}{\eps}\bigl(\lambda_1|V_0|+\lambda_2|V_0'|+\l|f\r|_{L^1(0,T)}\bigr)^{\! \!2}\right]+
  \frac{T\sigma^*|\Omega|}{\eps \ell^2}\int_\Omega{\sigma(\vartheta_0)\left|\nabla\phi_0\right|^2\,\pier{dv}}\\
  &\quad{}+\frac{1}{2\eps \ell}\int_0^t\int_\Omega{\sigma(\vartheta_{\tau}(r-\tau))\left|\nabla\phi_{\tau}(r)\right|^2\,\pier{dv}dr}+
  \frac{\eps}{2}\bigl(1+\frac{\sigma^*|\Omega|}{\pier{\ell}}\bigr)\int_0^t{V_{\tau}^2(r)\,dr}
  \end{split}
\]
for all $\eps>0$. If we make the choice $\eps= \displaystyle \frac{\lambda_2\ell }{\ell+\sigma^*|\Omega|}$, we infer that
\beq
  \begin{split}
  &\frac{\lambda_1}{2}\left|V_{\tau}(t)\right|^2+\frac{\lambda_2}{2}
   \int_0^t{V_{\tau}^2(r) \,dr} \\
  &\leq  \frac{\lambda_1}{2}V_0^2
  +\frac{T(\ell+\sigma^*|\Omega|)}{\lambda_2\ell}\bigl(\lambda_1
    |V_0|+\lambda_2|V_0'|+\l|f\r|_{L^1(0,T)}\bigr)^{\! 2}\\
  &\quad{}+ \frac{T\sigma^*|\Omega|(\ell +\sigma^*|\Omega| )}{\lambda_2 \ell^3}\int_ 
   \Omega{\sigma(\vartheta_0)\left|\nabla\phi_0\right|^2\,\pier{dv}}\\ 
  &\quad{}+ \frac{\ell+\sigma^*|\Omega|}{2\lambda_2 \ell^2}\int_0^t\int_
   \Omega{\sigma(\vartheta_{\tau}(r-\tau))\left|\nabla\phi_{\tau}(r)\right|^2\,\pier{dv}dr}
  \label{2est'}
  \end{split}
\eeq
for all $t\in[0,T]$. At this point, if we take $w= \psi_{\tau}(t)\in\gH1$ in \eqref{2'}, arguing as for \eqref{est_psi_V} we obtain
\beq
\label{2est'''}
  \int_{\Omega}{\sigma(\vartheta_{\tau}(t-\tau))\left|\nabla\psi_{\tau}(t)\right|^2\,\pier{dv}}\leq
  \int_{\Omega}{\sigma(\vartheta_{\tau}(t-\tau))\frac{V_{\tau}^2(t)}{\ell^2}\,\pier{dv}}\,.
\eeq
Therefore, with the help of \eqref{psi_u'} and \eqref{bound_sig_k} we easily deduce that
\beq
  \label{2est''}
  \int_{\Omega}{\sigma(\vartheta_{\tau}(t-\tau))\left|\nabla\phi_{\tau}(t)\right|^2\,\pier{dv}}\leq
  \frac{4\sigma^*|\Omega| } {\ell^2} V_{\tau}^2(t) .
\eeq
Now, if we substitute \eqref{2est''} in \eqref{2est'} and recall \eqref{pier4}, we 
have that 
\[
  \begin{split}
  \frac{\lambda_1}{2}\left|V_{\tau}(t)\right|^2
  &\leq  \frac{\lambda_1}{2}V_0^2
  +\frac{T(\ell+\sigma^*|\Omega|)}{\lambda_2\ell}\bigl(\lambda_1
    |V_0|+\lambda_2|V_0'|+\l|f\r|_{L^1(0,T)}\bigr)^{\! 2}\\
  &\quad{}+ \frac{T|\sigma^*|^2 |\Omega|(\ell +\sigma^*|\Omega| )}{\lambda_2 \ell^3} \,
  \|\phi_0 \|_{\H1}^2 + \frac{2\sigma^*|\Omega|( \ell+\sigma^*|\Omega|)}{\lambda_2 \ell^4}\int_0^t
  |V_{\tau} (r)|^2dr
  \end{split}
\]
for all $t\in[0,T]$. Hence, we can apply the Gronwall lemma and conclude that 
$\{V_\tau\}$ is bounded in $L^\infty(0,T)$ independently of $\tau$.
Moreover, taking this into account in \eqref{2est'''}-\eqref{2est''} and recalling that 
$\sigma (r) \geq\sigma_*>0$ for all $r\in\Ar$,
thanks to the Poincar\'e inequality and \eqref{psi_u'} we obtain the estimate~\eqref{3est}.
Finally, using condition \eqref{3est} and the bound for $\|V_\tau\|_{L^\infty(0,T)} $,
by a comparison of terms in \eqref{1'} we deduce \eqref{2est}.

\subsection{The third estimate}
We would like to show now there exists a contant $C_3>0$ such that 
\beq
  \label{est_inf1}
    \|\vartheta_\tau\|_{L^\infty(0,T;L^1(\Omega))} + 
  \|  u_\tau\|_{L^\infty(0,T;L^1(\Omega))}\leq C_3 \quad \text{for all } \tau \in (0,\tau^*)\,.
\eeq
We choose $w=1$ in \eqref{3'}.  Note that $\beta(u_\tau)\geq0$ almost everywhere on $\Sigma_{\pier{\ell}}$, thanks 
to \eqref{data2} and since
\beq u_\tau\geq u^*:= K(\vartheta_*)>0 \label{pier14} \eeq
by \eqref{1est} and \eqref{pier5'}. Then,  on account of \eqref{3est}  we deduce that
\[
  \int_\Omega{\frac{\partial\vartheta_\tau}{\partial t}(t)\,\pier{dv}}\leq
  \sigma^*(C_2)^2 +|\Gamma_{\pier{\ell}}|^{1/2}\l|h_\Gamma(t)\r|_{L^2(\Gamma_{\pier{\ell}})}\  \hbox{ for a.e. $t\in[0,T]$},
\]
where $ |\Gamma_{\pier{\ell}}|$ denotes the surface measure of $\Gamma_{\pier{\ell}}$.
Now, by integrating with respect to $t$ and using the H\"older inequality we obtain
\beq
  \int_\Omega{\vartheta_\tau(t)\,\pier{dv}}\leq
  \int_\Omega{\vartheta_0\,\pier{dv}}+T\sigma^*(C_2)^2 +\sqrt{|\Gamma_{\pier{\ell}}|T} \l|h_\Gamma\r|_{L^2(0,T;L^2(\Gamma_{\pier{\ell}}))}
\label{pier13}
\eeq
for almost every $t\in[0,T]$. Observe tht $\vartheta_0\in L^1(\Omega)$ by virtue of \eqref{pier3}. We also recall the properties \eqref{K}--\eqref{gamma} which imply that 
$ |K(r)|\leq C_K r $, $r\in \Ar$, for some Lipschitz constant $C_K$. Hence, the inequality 
\eqref{pier13} and condition \eqref{pier5'} imply \eqref{est_inf1}.

\subsection{The fourth estimate}
This is the really important estimate we are going to prove. For convenience,
let us ask the reader to state the estimate, unusually, at the end of the subsection.
We have been inspired by a computation already performed in \cite[pp.~1361--1362]{fpr}.

The idea is to test equation \eqref{3'} by 
\beq
  \label{w}
  w=1-\frac{1}{(1+u_{\tau}(t))^{1-\alpha}}\in\H1\,,  \ \hbox{ with } \, \alpha \in (0,1)\,:
\eeq
please note that this makes sense since \eqref{pier14} holds almost everywhere in $Q$.
Let us handle the different terms of \eqref{3'} separately. The first term is 
\[
  \begin{split}
  &\int_{\Omega}{\frac{\partial\vartheta_{\tau}}{\partial t}(t)\bigl(1-\frac{1}{(1+u_{\tau}(t))^{1-\alpha}}\bigr)\,\pier{dv}}=
  \int_\Omega{\frac{\partial\vartheta_{\tau}}{\partial t}(t)\bigl(1-\frac{1}{(1+K(\vartheta_{\tau}(t)))^{1-\alpha}}\bigr)\,\pier{dv}}\\
  &=\frac{d}{dt}\int_\Omega{F(\vartheta_{\tau}(t))\,\pier{dv}}\,, \quad\text{where } \, 
  F(r):=\int_0^r{\bigl(1-\frac{1}{(1+K(\rho))^{1-\alpha}}\bigr)\,d\rho}, \ \, r\geq 0; 
  \end{split}
\]
please note that this computation is formal, since the regularity \eqref{sol3_app} would not
allow us to apply the chain rule.  However, this is not restrictive since we can repeat the same argument on a further approximating scheme. Secondly, if we set
\beq
  \label{z_tau}
  z_{\tau}:=\bigl(1+u_{\tau}\bigr)^{\alpha/2}\,,
\eeq
the second term from equation \eqref{3'} can be handled as follows:
\[
  \begin{split}
  &\int_\Omega{\nabla u_{\tau}(t)\cdot\nabla\bigl(1-\frac{1}{(1+u_{\tau}(t))^{1-\alpha}}\bigr)\,\pier{dv}}\\
  &=-\int_\Omega{\nabla\bigl(1+u_{\tau}(t)\bigr)\cdot\nabla\bigl(1+u_{\tau}(t)\bigr)^{\alpha-1}\,\pier{dv}}
  =-\int_\Omega{\nabla z_{\tau}^{2/\alpha}(t)\cdot\nabla z_{\tau}^{2-2/\alpha}(t)\,\pier{dv}}\\
  &=\int_\Omega{-\frac{2}{\alpha}z_{\tau}^{-1 + 2/\alpha}(t)\nabla z_{\tau}(t)\cdot\bigl(2-\frac{2}{\alpha}\bigr)z_{\tau}^{1-2/\alpha}(t)\nabla z_{\tau}(t)\,\pier{dv}}
  =\int_{\Omega}{\frac{4(1-\alpha)}{\alpha^2}\left|\nabla z_{\tau}(t)\right|^2\,\pier{dv}}\,.
  \end{split}
\]
Moreover, the assumption \eqref{data2} and condition \eqref{pier14} imply that
\[
  \int_{\Gamma_{\pier{\ell}}}{\beta(u_{\tau}(t))\bigl(1-\frac{1}{(1+u_{\tau}(t))^{1-\alpha}}\bigr)\,ds}\geq0\,.
\]
Finally, as 
\[
  \left|1-\frac{1}{(1+u_{\tau}(t))^{1-\alpha}}\right|\leq1\,,
\]
with the help of  \eqref{3est} we deduce that
\begin{gather}
  \nonumber
  \int_{\Omega}{\sigma(\vartheta_{\tau}(t-\tau))T_\tau\bigl(\left|\nabla\phi_{\tau}(t)\right|^2\bigr)\bigl(1-\frac{1}{(1+u_{\tau}(t))^{1-\alpha}}\bigr)\,\pier{dv}}
  \leq \sigma^* (C_2)^2 \,,\\
  \nonumber
  \int_{\Gamma_{\pier{\ell}}}{h_{\Gamma}(t)\bigl(1-\frac{1}{(1+u_{\tau}(t))^{1-\alpha}}\bigr)\,ds}\leq\frac{1}{2}\l|h_{\Gamma}(t)\r|_{L^2(\Gamma_{\pier{\ell}})}^2
  +\frac{|\Gamma_{\pier{\ell}}|}{2}\,.
\end{gather}
Taking all these estimates into account, equation \eqref{3'} tested by \eqref{w} tells us that
\[
\begin{split}
  &\frac{d}{dt}\int_\Omega{F(\vartheta_{\tau}(t))\,\pier{dv}}+
  \int_{\Omega}{\frac{4(1-\alpha)}{\alpha^2}\left|\nabla z_{\tau}(t)\right|^2\,\pier{dv}}\\
  &\leq
  \sigma^* (C_2)^2 +\frac{1}{2}\l|h_{\Gamma}(t)\r|_{L^2(\Gamma_{\pier{\ell}})}^2+\frac{|\Gamma_{\pier{\ell}}|}{2} \quad \hbox{for a.e. } \, t\in [0,T]; 
\end{split}
\]
hence, integrating on $(0,T)$ we obtain
\[
  \frac{4(1-\alpha)}{\alpha^2}\int_{Q}{\left|\nabla z_{\tau}(t)\right|^2\,\pier{dv}dt}\leq
  \int_{\Omega}{F(\vartheta_0)\,\pier{dv}}+T\bigl(\sigma^*(C_2)^2+\frac{|\Gamma_{\pier{\ell}}|}{2}\bigr)+\frac{1}{2}\l|h_{\Gamma}\r|_{L^2(0,T;\L2)}^2\,.
\]
As $F$ grows at most linearly at infinity, from \eqref{pier3} and \eqref{data1} we deduce that there exists a positive constant $D_1$ such that
\beq
  \label{z_est1}
  \l|\nabla z_{\tau}\r|_{L^2(0,T;\L2)}\leq D_1 \quad\text{for all } \tau \in (0,\tau^*)\,;
\eeq
furthermore, \eqref{est_inf1} and \eqref{z_tau} imply that there exists a constant $D_2>0$ such that
\beq
  \label{z_est2}
  \l|z_{\tau}\r|_{L^\infty(0,T;L^{2/\alpha}(\Omega))}\leq D_2 \quad\text{for all } \tau \in (0,\tau^*)\,.
\eeq
Now, by \eqref{z_est1} and \eqref{z_est2} $\{z_\tau\}$ is bounded in $L^2 (0,T;\H1)$ independently of $\tau$ and, as $\Omega\subseteq\Ar^3$, the Sobolev embedding theorem entails that
\beq
  \label{z_est3}
  \l|z_{\tau}\r|_{L^2(0,T;L^6(\Omega))}\leq D_3 \quad\text{for all } \tau \in (0,\tau^*),
\eeq
for a certain positive constant $D_3$. At this point, we are able to use the following interpolation inequality:
\beq
\begin{split}
  \label{interp}
  \l|z_{\tau}\r|_{L^r(0,T;L^s(\Omega))}\leq
  \l|z_{\tau}\r|_{L^\infty(0,T;L^{2/\alpha}(\Omega))}^{\vartheta}
  \l|z_{\tau}\r|_{L^2(0,T;L^6(\Omega))}^{1-\vartheta}\,,\\[0.2cm]
  \hbox{with } \ \frac{1}{r}=\frac{\vartheta}{\infty}+\frac{1-\vartheta}{2}\,,\ \ 
  \frac{1}{s}=\frac{\alpha\vartheta}{2}+\frac{1-\vartheta}{6}\,.
\end{split}
\eeq
Making the particular choice $r=s$ leads to $\vartheta= 2/(2+3\alpha) $ and consequently, due to \eqref{z_est2} and \eqref{z_est3}, there exists $D_4>0$ such that
\beq
  \label{z_est4}
  \l|z_\tau\r|_{L^{\frac{2(2+3\alpha)}{3\alpha}}(Q)}\leq D_4 \quad\text{for all } \tau \in (0,\tau^*)\,.
\eeq
Recalling now \eqref{z_tau}, condition \eqref{z_est4} implies
\beq
  \label{u_est}
  \l|u_\tau\r|_{L^{\frac{2+3\alpha}{3}}(Q)}\leq D_5 \quad\text{for all } \tau \in (0,\tau^*)\,,
\eeq
for a certain positive constant $D_5$. At this point, thanks to the generalized Hölder inequality and conditions \eqref{z_est1} and \eqref{z_est4},
 we note that for almost all $t\in[0,T]$
\[
  \nabla u_\tau(t)=\frac{2}{\alpha}z_\tau^{\frac{2-\alpha}{\alpha}}(t)\nabla z_\tau(t)\in L^{p}(\Omega)\,,
\]
where $p$ is the conjugate exponent of $q$ such that
\[
  \frac{3\alpha}{2(2+3\alpha)}\frac{2-\alpha}{\alpha}+\frac{1}{2}+\frac{1}{q}=1\,;
\]
then, an easy calculation shows that for all values $\alpha\in (2/3,1)$
\beq
  \label{p}
  q=\frac{3\alpha+2}{3\alpha-2}\in (5,+\infty)\,,\quad p=\frac{2+3\alpha}{4}\in (1,5/4)  \,.
\eeq
Taking all these remarks into account, we deduce that there exists $D_6>0$ such that
\beq
  \label{u_est'}
  \l|u_\tau\r|_{L^p(0,T;W^{1,p}(\Omega))}\leq D_6 \quad \text{for all } \tau \in (0,\tau^*)\,.
\eeq
Recalling now \eqref{pier5'} and the fact that $\gamma$ is Lipschitz continuous, 
the two estimates \eqref{u_est} and \eqref{u_est'} along with the position \eqref{p}
ensure that there exists a positive constant $C_4$ such that
\beq
\begin{split}
  \label{4est}
  \l|u_\tau\r|_{L^{4p/3}(Q)\cap L^p(0,T;W^{1,p}(\Omega))} + \l|\vartheta_\tau\r|_{L^{4p/3}(Q)\cap L^p(0,T;W^{1,p}(\Omega))}\leq C_4\\ \quad\text{for all } \tau \in (0,\tau^*)\,.
  \end{split} 
\eeq

\subsection{The fifth estimate}
We now prove that there is a positive constant $C_5$ such that
\beq
\label{5est}
  \l|\frac{\partial\vartheta_\tau}{\partial t}\r|_{L^p(0,T;W^{1,q}(\Omega)')}\leq C_5 \quad\text{for all } \tau \in (0,\tau^*)\,,
\eeq
where $p$ and $q$ are the coniugate exponents introduced in \eqref{p}. 

We proceed by taking $w\in W^{1,q}(\Omega)$ in 
equation \eqref{3'}: isolating the first term, thanks to the Hölder inequality
and conditions \eqref{bound_sig_k} and \eqref{beta_lin}, we obtain
\[
  \begin{split}
  \int_{\Omega}{\frac{\partial\vartheta_{\tau}}{\partial t}(t)w\,\pier{dv}}&=
  -\int_{\Omega}{\nabla u_{\tau}(t)\cdot\nabla w\,\pier{dv}}
  -\int_{\Gamma_{\pier{\ell}}}{\beta(u_{\tau}(t))w\,ds}\\
  &\quad{}+\int_{\Omega}{\sigma(\vartheta_{\tau}(t-\tau))T_\tau\bigl(\left|\nabla\phi_{\tau}(t)\right|^2\bigr)w\,\pier{dv}}
  +\int_{\Gamma_{\pier{\ell}}}{h_{\Gamma}(t)w\,ds}\\
  &\leq\l|\nabla u_\tau(t)\r|_{L^p(\Omega)}\l|\nabla w\r|_{L^q(\Omega)}+C_\beta \l|w\r|_{L^1(\Gamma_{\pier{\ell}})}+C_\beta \l|u_\tau(t)\r|_{L^p(\Gamma_{\pier{\ell}})}
  \l|w\r|_{L^q(\Gamma_{\pier{\ell}})}\\
  &\quad{}+ \sigma^* \l| \phi_{\tau}\r|^2_{L^\infty(0,T;\H1)}
  \l|w\r|_{L^\infty(\Omega)}+\l|h_{\Gamma}(t)\r|_{L^2(\Gamma_{\pier{\ell}})}\l|w\r|_{L^2(\Gamma_{\pier{\ell}})}\,.
  \end{split}
\]
Recall that for all $r\in[1,+\infty)$ there exists a constant $M_r>0$ such that
\beq
  \label{trace}
  \l|w\r|_{L^r(\Gamma_{\pier{\ell}})}\leq M_r\l|w\r|_{W^{1,r}(\Omega)} \quad\text{for all } w\in W^{1,r}(\Omega)
\eeq
(see, e.g., \cite{BG87}) and that the following embedding 
$
  W^{1,q}(\Omega)\subseteq L^\infty(\Omega)
$
is continuous since $q >3$ and $\Omega\subseteq\Ar^3$; hence, using \eqref{3est} and \eqref{u_est'}, from the previous calculations we deduce that
\[
  \int_{\Omega}{\frac{\partial\vartheta_{\tau}}{\partial t}(t)w\,\pier{dv}}\leq
  \bigl(\l|u_\tau(t)\r|_{W^{1,p}(\Omega)}+D_7+ D_8 \l|h_{\Gamma}(t)\r|_{L^2(\Gamma_{\pier{\ell}})}\bigr)\l|w\r|_{W^{1,q}(\Omega)}\,,
\]
for all $w\in W^{1,q}(\Omega)$ and for some constants $D_7,\, D_8$ independent of $\tau$. 
Then, we have that 
\[
  \l|\frac{\partial\vartheta_\tau}{\partial t}(t)\r|_{W^{1,q}(\Omega)'}\leq
 \l|u_\tau(t)\r|_{W^{1,p}(\Omega)}+D_7+ D_8 \l|h_{\Gamma}(t)\r|_{L^2(\Gamma_{\pier{\ell}})}
\]
for almost all $t\in[0,T]$; taking the $p$th power of this 
expression and integrating on $(0,T)$, as $p<2$
and \eqref{u_est'} holds we infer
\[
  \l|\frac{\partial\vartheta_\tau}{\partial t}\r|_{L^p(0,T;W^{1,q}(\Omega)')}\leq
  D_6+ T^{1/p}D_7 + D_8 T^{1/p-1/2}\l|h_\Gamma\r|_{L^2(0,T;L^2(\Gamma_{\pier{\ell}}))}
\]
and \eqref{5est} follows easily from the above inequality.


\section{The passage to the limit}
\setcounter{equation}{0}
\label{limit}

In this last section, we collect all the arguments needed to pass to the limit in the \pier{approximating} problem \eqref{1'}--\eqref{init'}
as $\tau\searrow0$ and recover in this way a solution of the original problem \eqref{1}--\eqref{init}.

\subsection{The terms in $\vartheta_\tau$ and $u_\tau$}
Firstly, let us point out that
\[
  W^{1,p}\bigl(\Omega\bigr)\subseteq\subseteq W^{1-\eps,p}\bigl(\Omega\bigr)\subseteq W^{1,q}\bigl(\Omega\bigr)' \quad \hbox{for all }\, \eps \in (0,1],  
\]
where the symbol $\subseteq\subseteq$ indicates a compact inclusion: hence, conditions \eqref{4est} and \eqref{5est}
ensure that we are able to apply the result contained in \cite[Cor.~4, p.~85]{simon} and infer that $\{\vartheta_\tau\}_{\tau \in (0,\tau^*)}$ is 
relatively compact in $L^p(0,T;W^{1-\eps,p}(\Omega))$. It follows that there exists
$\vartheta$  and a subsequence $\{\tau_n\}_{n\in\En}$ such that
\beq
  \label{conv_teta}
  \vartheta_{\tau_n}\rarr\vartheta \quad\text{in } L^p(0,T;W^{1-\eps,p}(\Omega)) 
  \quad \text{as } n\rarr\infty\,.
\eeq
Hence, we let $1- \eps >1/p$, that is, $ \eps < 1-1/p$,  and observe that, due
to the trace theory for Sobolev spaces (see, e.g., \cite[Thm.~2.24, p.~1.61]{BG87}), we have the convergence
\beq
  \label{conv_teta-pier}
  \vartheta_{\tau_n}\rarr\vartheta \, \ \text{in } L^p(Q) \, \hbox{ and, for the traces, } \, 
  \vartheta_{\tau_n}|_{\Gamma_\ell} \rarr \vartheta|_{\Gamma_\ell}\, \text{ in } L^p(\Sigma_\ell).
\eeq
Moreover, as $u_{\tau_n} = \gamma^{-1} (\vartheta_{\tau_n})$ and $\gamma^{-1} 
=K$ is Lipschitz continuous (see \eqref{pier5'} and \eqref{gamma}),  we also have 
\beq
  \label{conv_u-pier}
  u_{\tau_n}\rarr u : = K(\vartheta) \, \ \text{in } L^p(Q) \, \hbox{ and, for the traces, } \, 
  u_{\tau_n}|_{\Gamma_\ell} \rarr u|_{\Gamma_\ell}\, \text{ in } L^p(\Sigma_\ell)
\eeq
as $n\to \infty.$ 
Furthermore, since condition \eqref{4est} implies weak compactness, we also have that
\beq
  \label{conv-teta-u}
 \vartheta_{\tau_n}\rarrw\vartheta  \, \hbox{ and } \, 
  u_{\tau_n}\rarrw u
 \quad\text{in }\, L^{4p/3}(Q)\cap L^p(0,T;W^{1,p}(\Omega))
  \quad\text{as } n\rarr\infty\,.
\eeq
On the other hand, the estimate \eqref{5est} implies that there exists
\[
  \eta\in L^p\bigl(0,T;W^{1,q}(\Omega)'\bigr)
\]
such that (in principle for a subsequence)
\beq
  \label{conv_eta}
  \frac{\partial\vartheta_{\tau_n}}{\partial t}\rarrw\eta \quad\text{in } L^p\bigl(0,T;W^{1,q}(\Omega)'\bigr)
  \quad\text{as } n\rarr\infty\,.
\eeq
Now, as $\vartheta_{\tau_n}\in W^{1,p}(0,T;W^{1,q}(\Omega)')$ for all $n$, we have that
\[
  \int_0^T{\left<\vartheta_{\tau_n}(t),\varphi'(t)\right>_{W^{1,q}(\Omega)}\,dt}=
  -\int_0^T{\left<\frac{\partial\vartheta_{\tau_n}}{\partial t}(t),\varphi(t)\right>_{W^{1,q}(\Omega)}dt}
\]
for all $\varphi$ say in $ \mathscr{D}\bigl(0,T;W^{1,q}(\Omega)\bigr);$
then, letting $n\rarr\infty$ and using \eqref{conv-teta-u}--\eqref{conv_eta} we deduce that
\[
  \int_0^T{\left<\vartheta(t),\varphi'(t)\right>_{W^{1,q}(\Omega)}\,dt}=
  -\int_0^T{\left<\eta(t),\varphi(t)\right>_{W^{1,q}(\Omega)}\,dt}
  \quad\forall\, \varphi\in\mathscr{D}\bigl(0,T;W^{1,q}(\Omega)\bigr)\,.
\]
Hence, it follows that $\eta = \frac{\partial\vartheta}{\partial t}$, that is, 
$\vartheta\in W^{1,p}\bigl(0,T;W^{1,q}(\Omega)'\bigr) $ and 
\beq
  \label{conv3_teta}
  \vartheta_{\tau_n}\rarrw\vartheta \quad\text{in } W^{1,p}\bigl(0,T;W^{1,q}(\Omega)'\bigr)
  \quad\text{as } n\rarr\infty\,.
\eeq
Next, we recall \eqref{conv_u-pier} and
for the traces (now, denoted again without the symbol $|_{\Gamma_\ell}$) we conclude that, possibly extracting another subsequence, 
\beq
  \label{conv4_u}
  u_{\tau_n}\rarr u \quad\text{a.e. in } \Sigma_{\pier{\ell}}
\eeq
which implies that
\beq
  \label{conv_beta}
  \luca{\beta(u_{\tau_n})\rarr \beta(u) \quad\text{a.e. in } \Sigma_{\pier{\ell}} }
\eeq
as $n\to \infty.$ 
Secondly, conditions \eqref{conv_u-pier} and \eqref{beta_lin} ensure that the sequence $\{\beta(u_{\tau_n})\}_{n\in\En}$ is bounded in $L^p(\Sigma_{\pier{\ell}})$; 
consequently, since $T|\Gamma_{\pier{\ell}}|<+\infty$, we can apply the 
Severini-Egorov theorem to check that \eqref{conv_beta} entails 
\beq
  \label{conv3_beta}
  \beta(u_{\tau_n})\rarr\beta(u) \quad\text{in }\,  L^r\bigl(\Sigma_{\pier{\ell}}\bigr) 
  \, \text{ for all } r \in [1,p) 
\eeq
and   
\beq
   \label{conv2_beta}
  \beta(u_{\tau_n})\rarrw\beta(u) \quad\text{in } L^p\bigl(\Sigma_{\pier{\ell}}\bigr) 
  \quad\text{as } n\rarr\infty\,.\\ 
\eeq

\subsection{The terms in $V_\tau$, $\psi_\tau$ and $\phi_\tau$}
In analogy with the arguments used in \eqref{conv-teta-u}--\eqref{conv3_teta} and in view of 
\eqref{2est}--\eqref {3est}, by weak star compactness we infer the existence of 
\beq
\nonumber  V\in W^{1,\infty} (0,T) , \quad 
  \psi\in L^\infty\bigl(0,T;\gH1\bigr), \quad \phi\in L^\infty\bigl(0,T;\H1\bigr)
\eeq
such that
\begin{gather}
  \label{conv_V}
  V_{\tau_n}\weakstar V \quad\text{in } W^{1,\infty}  (0,T)\, , \\
  \label{conv_psi}
  \psi_{\tau_n}\weakstar \psi \quad\text{in } L^\infty\bigl(0,T;\gH1\bigr), \\
  \label{conv_phi}
  \phi_{\tau_n}\weakstar\phi \quad\text{in } L^\infty\bigl(0,T;\H1\bigr)
\end{gather}
as $n \to \infty$. In principle, one should possibly extract another subsequence, but let us say 
that the subsequence $\tau_n$ can be selected at the same time as for \eqref{conv-teta-u},  
\eqref{conv3_teta} and \eqref{conv_V}--\eqref{conv_phi}.  Thanks to \eqref{conv_V}, we 
are able to apply the Ascoli theorem (see, e.g., \cite[Lem.~1, p.~71]{simon}) and deduce 
that 
\beq
  \label{conv2_V}
  V_{\tau_n}\rarr V \quad \text{in } C^0([0,T]) \quad\text{as } n\rarr\infty\,.
\eeq

We aim to prove that the quintuplet $(V,\phi, \psi, \vartheta, u)$ solves the problem \eqref{1}--\eqref{init}. Then, 
we start to check  that the limit functions $V$ and $\vartheta$ satisfy the initial conditions  \eqref{init}. It is clear that $V(0)=V_0$ thanks to \eqref{init'} and \eqref{conv2_V}.
Moreover, condition \eqref{conv3_teta} ensures that 
\[
  \left<\vartheta_{\tau_n}(0),w\right>_{W^{1,q}(\Omega)}\rarr\left<\vartheta(0),w\right>_{W^{1,q}(\Omega)} 
  \quad\text{for all } w\in W^{1,q}(\Omega)\,, \quad\text{as } n\rarr\infty\,,
\]
from which we also deduce that $\vartheta(0)=\vartheta_0$. Moreover, we point out that 
\eqref{psi_u} is satisfied due to \eqref{psi_u'} and \eqref{conv_V}--\eqref{conv_phi},
and \eqref{pier5} follows from \eqref{pier5'} and \eqref{conv_teta-pier}--\eqref{conv_u-pier}. 
Hence, it remains to check \eqref{1}--\eqref{3}: for this aim, we have to 
pass to the limit in the nonlinear coupling terms of \eqref{1'}--\eqref{3'}. 
First, we consider the term $\sigma(\vartheta_{\tau}(t-\tau))$ and prove through some intermediate steps that
\beq
  \label{conv_sigma}
  \sigma\bigl(\vartheta_{\tau_n}(\cdot-\tau_n)\bigr)\rarr\sigma\bigl(\vartheta\bigr) \quad \text{in }\,  L^r(Q)
\, \text{ for all } r\in[1,+\infty)\,,    \quad\text{as } n\rarr\infty\,.
\eeq
\begin{lem}
  \label{lemma1}
  In the current hypotheses, by extending $\vartheta$ as in \eqref{primadi0} we have that
  \beq
    \label{conv_tetatau}
    \vartheta(\cdot-\tau_n)\rarr\vartheta \quad\text{in } L^p(Q) \quad \text{as } n\rarr\infty\,.
  \eeq
\end{lem}
\begin{proof}
Since $C^0(\overline{Q})$ is dense in $L^p(Q)$, for all $\eps>0$ there exists $\eta_\eps\in C^0(\overline{Q})$ such that
\[
  \l|\eta_\eps-\vartheta\r|_{L^p(Q)}\leq\eps\,.
\]
Moreover, if we let $\luca{\eta_\epsilon (t) = \vartheta_0 }$ for $t<0$, we can write
\[
  \begin{split}
  &\l|\vartheta(\cdot-\tau_n)-\vartheta\r|_{L^p(Q)}\\
  &\leq
  \l|\vartheta(\cdot-\tau_n)-\eta_\eps(\cdot-\tau_n)\r|_{L^p(Q)}
  +\l|\eta_\eps(\cdot-\tau_n)-\eta_\eps\r|_{L^p(Q)}+
  \l|\eta_\eps-\vartheta\r|_{L^p(Q)}. 
  \end{split}
\]
Now, the first and the third term on the right hand side of the previous expression are 
under control of $\eps$,  and for such fixed $\eps >0 $ and for all $n$ 
sufficiently large the second term in the right hand side is bounded by $\eps$
as well (recall that $\vartheta_0 \in \L2 \subset L^p(\Omega) $ and use the Lebesgue 
dominated convergence theorem, for example).
\end{proof}

\begin{lem}
  \label{lemma2}
  In the current hypotheses, we have that
  \beq
    \label{conv-teta-utau}
    \vartheta_{\tau_n}(\cdot-\tau_n)\rarr\vartheta \quad\text{in } L^p(Q)\,, \quad\text{as } n\rarr\infty\,.
  \eeq
\end{lem}
\begin{proof}
An easy calculation shows that
\[
  \begin{split}
  \l|\vartheta_{\tau_n}(\cdot-\tau_n)-\vartheta\r|_{L^p(Q)}&\leq
  \l|\vartheta_{\tau_n}(\cdot-\tau_n)-\vartheta(\cdot-\tau_n)\r|_{L^p(Q)}+
  \l|\vartheta(\cdot-\tau_n)-\vartheta\r|_{L^p(Q)}\\
  &\leq\l|\vartheta_{\tau_n}-\vartheta \r|_{L^p(Q)}+
  \l|\vartheta(\cdot-\tau_n)-\vartheta \r|_{L^p(Q)}\rarr0
  \end{split}
\]
as $n\rarr\infty$, using conditions \eqref{conv_teta-pier} and \eqref{conv_tetatau}.
\end{proof}

We are now ready to prove \eqref{conv_sigma}.
Indeed, condition \eqref{conv-teta-utau} implies that there is a further subsequence of $\{\tau_n\}_{n\in\En}$, for which we keep the same notation, such that
\[
  \vartheta_{\tau_n}(\cdot-\tau_n)\rarr\vartheta\quad\text{a.e. in } Q \,;
\]
on account of \eqref{bound_sig_k}, the continuity of $\sigma$ implies that
\beq
  \label{conv2_sigma}
  \sigma\bigl(\vartheta_{\tau_n}(\cdot-\tau_n)\bigr)\rarr\sigma\bigl(\vartheta \bigr) \quad\text{a.e. in } Q\,, \quad\text{as } n\rarr\infty,
\eeq
and this convergence and the boundedness of $\sigma$ ensure that we can apply the dominated convergence theorem and thus deduce \eqref{conv_sigma}.

The information we have collected are sufficient to pass to the limit in equation \eqref{2'}: 
indeed, we note that the sequence $\{ \sigma(\vartheta_{\tau_n}(\cdot-\tau_n))\nabla\psi_{\tau_n} \}$ is bounded in $L^\infty \bigl( 0,T; \L2^3\bigr)$ (cf.~\eqref{3est}
and \eqref{bound_sig_k}) and weakly converges to 
$ \sigma(\vartheta)\nabla\psi $, e.g., in $ L^1\bigl( 0,T; L^p(\Omega)^3\bigr) $ 
(cf.~ \eqref{conv_sigma} and \eqref{conv_psi}), so that actually the weak star convergence $\sigma(\vartheta_{\tau_n}(\cdot-\tau_n))\nabla\psi_{\tau_n} \weakstar \sigma(\vartheta)\nabla\psi$  holds in $L^\infty \bigl( 0,T; \L2^3\bigr)$. On the other hand, from  \eqref{conv_sigma} and \eqref{conv2_V} it follows that
\beq
  \label{conv_sig_V}
  \sigma\bigl(\vartheta_{\tau_n}(\cdot-\tau_n)\bigr) V_{\tau_n}\rarr\sigma\bigl(\vartheta\bigr) V \quad \text{in }\,  L^r(0,T; L^r(\Omega)),
\, \text{ for all } r\in[1,+\infty)\,.
\eeq
Then, we can pass to the limit in \eqref{2'} (or in some equivalent time-integrated version of it,
with test function $w \in  L^1 \bigl(0,T;\gH1\bigr))$ to obtain \eqref{2}.

Finally, let us complete this subsection proving a stronger convergence of $\{\psi_{\tau_n}\}_{n\in\En}$.
\begin{lem}
  \label{lemma2-pier}
  In the current hypotheses, as $n \to \infty $ we have that
\begin{gather}
  \label{conv2_psi}
  \psi_{\tau_n}\rarr\psi \quad\text{in } L^2\bigl(0,T;\gH1\bigr)
  \,,\\
  \label{conv2_phi}
  \phi_{\tau_n}\rarr\phi \quad\text{in } L^2\bigl(0,T;\H1\bigr)\,.
\end{gather}
\end{lem}
\begin{proof}
As a matter of fact, taking the difference between equations \eqref{2'} (written for $\tau_n$) 
and \eqref{2}, and testing it by $w=\psi_{\tau_n}(t)-\psi(t)\in\gH1$,
after a simple calculation, for almost all $t\in[0,T]$ we obtain that 
\[
  \begin{split}
  &\int_\Omega{\sigma\bigl(\vartheta_{\tau_n}(t-\tau_n)\bigr)
  \left|\nabla(\psi_{\tau_n}-\psi)(t)\right|^2\,\pier{dv}}\\
  &+\int_\Omega{\left[\sigma\bigl(\vartheta_{\tau_n}(t-\tau_n)\bigr)-\sigma\bigl(\vartheta(t)\bigr)\right]\nabla\psi(t)\cdot\nabla
  \bigl(\psi_{\tau_n}(t)-\psi(t)\bigr)\,\pier{dv}}\\
  &+\int_\Omega{\left[\sigma\bigl(\vartheta_{\tau_n}(t-\tau_n)\bigr)V_{\tau_n}(t)-\sigma\bigl(\vartheta(t)\bigr)V(t)\right]
  \frac{\partial}{\partial z}\bigl(\psi_{\tau_n}(t)-\psi(t)\bigr)\,\pier{dv}}=0\,.
  \end{split}
\]
Hence, using the Young inequality and condition \eqref{bound_sig_k} we deduce that
\[
  \begin{split}
  \sigma_*\l|\nabla(\psi_{\tau_n}-\psi)(t)\r|^2_{L^2(\Omega)^3} &\leq
  \eps\l|\nabla(\psi_{\tau_n}-\psi)(t)\r|^2_{\L2^3}\\
  &\quad{}+\frac{1}{2\eps}\int_{\Omega}{\left|\sigma\bigl(\vartheta_{\tau_n}(t-\tau_n)\bigr)
  -\sigma\bigl(\vartheta(t)\bigr)\right|^2\left|\nabla\psi(t)\right|^2\,\pier{dv}}\\
  &\quad{}+\frac{1}{2\eps}\int_\Omega{\left|\sigma\bigl(\vartheta_{\tau_n}(t-\tau_n)\bigr)V_{\tau_n}(t)-\sigma\bigl(\vartheta(t)\bigr)V(t)\right|^2\,\pier{dv}}
  \end{split}
\]
for all $\eps>0$. If we make the particular choice $\eps=\sigma_*/2$ and integrate on $(0,T)$, we have that
\[
  \begin{split}
  \frac{\sigma_*}{2}\l|\nabla (\psi_{\tau_n}-\psi)\r|^2_{L^2(0,T;L^2(\Omega))}&\leq
  \frac{1}{\sigma_*}\int_{Q}{\left|\sigma\bigl(\vartheta_{\tau_n}(\cdot -\tau_n)\bigr)
  -\sigma\bigl(\vartheta \bigr)\right|^2\left|\nabla\psi \right|^2\,\pier{dv}dt}\\
  &\quad +\frac{1}{\sigma_*}\int_Q{\left|\sigma\bigl(\vartheta_{\tau_n}(\cdot -\tau_n)\bigr)
  V_{\tau_n}-\sigma\bigl(\vartheta\bigr)V\right|^2\,\pier{dv}dt} .
  \end{split}
\]
Thanks to conditions \eqref{conv2_sigma} and \eqref{conv_sig_V}, using the dominated convergence theorem we see that
the right hand side of the above inequality goes to $0$ as $n\rarr\infty$. 
Therefore, from the Poincar\'e inequality
we infer condition \eqref{conv2_psi}. At this point,  \eqref{conv2_phi} turns out to be a straightforward consequence of \eqref{psi_u'} and \eqref{conv2_V}. 
\end{proof}

\subsection{Conclusion of the proof}
By virtue of \eqref{conv2_phi} and \eqref{conv_sigma}, now we can take the limit in the nonlinear term of equation \eqref{1'} and obtain \eqref{1}, using \eqref{conv_V}
to pass to the limit in the left hand side. To deal with \eqref{3'},
we need to show that 
\beq
  \label{conv_trunc}
  \sigma\bigl(\vartheta_{\tau_n}(\cdot-\tau_n)\bigr)T_{\tau_n}\bigl(\left|\nabla\phi_{\tau_n}\right|^2\bigr)\rarr
  \sigma\bigl(\vartheta\bigr)\left|\nabla\phi\right|^2 \quad\text{in } L^1(Q) \quad\text{as } n\rarr\infty\,.
\eeq
To this aim, we take advantage of the property stated in the next lemma.
\begin{lem}
  \label{lemma3}
  Let $T_\tau$ the truncation operator defined in \eqref{trunc} and let $\{z_n\}_{n\in\En}\subseteq L^1(Q)$ such that $
    z_n\rarr z $ in $L^1(Q).$ Then, it turns out that $ T_{\tau_n}(z_n)\rarr z $ in $ L^1(Q)$ as $ n\rarr\infty$.
\end{lem}
\begin{proof}
Thanks to the Lipschitz continuity of $T_{\tau_n}$ we have
\[
  \left|T_{\tau_n}(z_n)-z\right|\leq\left|T_{\tau_n}(z_n)-T_{\tau_n}(z)\right|+\left|T_{\tau_n}(z)-z\right|\leq
  \left|z_n-z\right|+\left|T_{\tau_n}(z)-z\right|
\]
and the dominated convergence theorem enables us to conclude the proof. 
\end{proof}
At this point, we observe that
\[
  \begin{split}
  &\left| \sigma\bigl(\vartheta_{\tau_n}(\cdot-\tau_n)\bigr)T_{\tau_n}\bigl(\left|\nabla\phi_{\tau_n}\right|^2\bigr)-
  \sigma\bigl(\vartheta\bigr)\left|\nabla\phi\right|^2\right| \\
  &\leq
  \sigma\bigl(\vartheta_{\tau_n}(\cdot-\tau_n)\bigr)\left|T_{\tau_n}\bigl(\left|\nabla\phi_{\tau_n}\right|^2\bigr)-\left|\nabla\phi\right|^2\right| +\left|\nabla\phi\right|^2\left|\sigma\bigl(\vartheta_{\tau_n}(\cdot-\tau_n)\bigr)-\sigma\bigl(\vartheta\bigr)\right|\,;
  \end{split}
\]
by integrating on $(0,T)$ the above inequality, the first term on the right hand side
goes to $0$ as $n\rarr\infty$  owing to the boundedness of  $\sigma$ and
 Lemma~\ref{lemma3}, while the second term tends to $0$ by virtue of \eqref{conv2_sigma}.
Hence, \eqref{conv_trunc} is proved.

Finally, letting $n\rarr\infty$ in \eqref{3'} and 
using conditions \eqref{conv-teta-u}, \eqref{conv3_teta}, \eqref{conv2_beta} and \eqref{conv_trunc}, we show exactly that equation \eqref{3} is satisfied. 


\end{document}